\documentclass[final,onefignum,onetabnum]{siamart190516}
\usepackage[utf8]{inputenc}
\usepackage{geometry, graphicx,wrapfig}
\usepackage{enumerate}
\usepackage{amsmath,amssymb,amsfonts,bm}
\usepackage{xcolor} 
\usepackage[linesnumbered,ruled,vlined,algo2e]{algorithm2e}
\usepackage[toc,page]{appendix}
\usepackage{makecell}
\usepackage{cleveref}

\newtheorem{assump}[theorem]{MP Setting}

\newcommand{\R}{\mathbb{R}}
\newcommand{\F}{\mathbb{F}}
\newcommand{\dd}{\delta}
\newcommand{\tth}{\theta}
\newcommand{\bb}[1]{\mathbf{#1}}
\newcommand{\fl}{\mathrm{fl}}
\newcommand{\cO}{\mathcal{O}}

\crefname{algocf}{alg.}{algs.}
\Crefname{algocf}{Algorithm}{Algorithms}

\title{Rounding Error Analysis of Mixed Precision Block Householder QR Algorithms\thanks{Submitted to editors October 30, 2019, ac\funding{This work was performed under the auspices of the U.S. Department of Energy by Lawrence Livermore National Laboratory under Contract DE-AC52-07NA27344 and was supported by the LLNL-LDRD Program under Project No. 17-SI-004, LLNL-JRNL-795525.}}}
\author{L. Minah yang\thanks{Department of Applied Mathematics, University of Colorado Boulder
		(\email{lucia.yang@colorado.edu}).}
	\and Alyson Fox\thanks{Center for Applied Scientific Computing, Lawrence Livermore National Laboratory, Livermore, CA 94551
		(\email{fox33@llnl.gov},\email{sanders29@llnl.gov}).}
	\and Geoffrey Sanders\footnotemark[3]}

\begin{document}

\maketitle
\begin{abstract}
	Although mixed precision arithmetic has recently garnered interest for training dense neural networks, many other applications could benefit from the  speed-ups and lower storage cost if applied appropriately. 
	The growing interest in employing mixed precision computations motivates the need for rounding error analysis that properly handles behavior from mixed precision arithmetic.
	We develop mixed precision variants of existing Householder QR algorithms and show error analyses supported by numerical experiments.	
\end{abstract}
\section{Introduction}\label{sec:intro}
The accuracy of a numerical algorithm depends on several factors, including numerical stability and well-conditionedness of the problem, both of which may be sensitive to rounding errors, the difference between exact and finite precision arithmetic. 
Low precision floats use fewer bits than high precision floats to represent the real numbers and naturally incur larger rounding errors. 
Therefore, error attributed to round-off may have a larger influence over the total error and some standard algorithms may yield insufficient accuracy when using low precision storage and arithmetic.
However, many applications exist that would benefit from the use of low precision arithmetic and storage that are less sensitive to floating-point round-off error, such as training dense neural networks \cite{micikevicius2018mixed} or clustering or ranking graph algorithms \cite{vonLuxburg2007}.
As a step towards that goal, we investigate the use of mixed precision arithmetic for the QR factorization, a widely used linear algebra routine.
\par

Many computing applications today require solutions quickly and often under low size, weight, and power constraints, such as in sensor formation, where low precision computation offers the ability to solve many problems with improvement in all four parameters.
Utilizing mixed precision, one can achieve similar quality of computation as high precision and still achieve 
speed, size, weight, and power constraint improvements. 
There have been several recent demonstrations of computing using IEEE half precision (fp16) achieving around half an order to an order of magnitude improvement of these categories in comparison to single and double precision (fp32, fp64).
Additionally, there exist demonstrations that the power consumption improvement is similar
\cite{fagan2016powerwall}.
Modern accelerators (e.g., GPUs, Knights Landing, or Xeon Phi) are able to achieve this factor or better speedup improvements.
Several examples include:
(i)   2-4$\times$ speedup in solving dense large linear equations \cite{haidar2018iterative,haidar2019tensorcore},
(ii)  12$\times$ speedup in training dense neural networks,
and
(iii) 1.2-10$\times$ speedup in small batched dense matrix multiplication \cite{abdelfattah2019batched} (up to 26$\times$ for batches of tiny matrices).
Training deep artificial neural networks by employing lower precision arithmetic to various tasks such as multiplication \cite{Courbariaux2014Mult} and storage \cite{Courbariaux2014Storage} can easily be implemented on GPUs and are a common practice in some data science applications.\par

The low precision computing environments that we consider are \emph{mixed precision} settings, which are designed to imitate those of new GPUs that employ multiple precision types for certain tasks. 
For example, Tesla V100's TensorCores perform block Fused Multiply Add operations (bFMAs), where matrix products of fp16 input data can be computed up to $16\times$ faster than that of fp64.
Many existing rounding error analyses of linear algebra routines are built within what we call a \emph{uniform precision} setting, which is the assumption that all arithmetic operations and storage are performed via the same precision.
In this work, we develop mixed precision variants of existing Householder (HH) QR factorization algorithms and perform mixed precision error analysis.
This work focuses on analyzing a few algorithms that use fp16/fp32 as the low/high precision types, but the error analysis can be easily modified for different floating point types (such as bfloat16 in \cite{tagliavini2018floating}).
The standard HH QR algorithm (HQR) and its block variants that partition the columns (level-3 BLAS variant, see \cite{golub2013matrix,Higham2002}) and those that partition the rows (communication-avoiding algorithms of \cite{Demmel2012}) are presented in \cref{sec:algo}, then modified to support bFMAs and an ad hoc mixed precision setting that mimics NVIDIA TensorCores in \cref{sec:mpanalysis}.
Our key findings are that mixed precision error analyses produce tighter error bounds as supported by experiments in \cref{sec:NE}, algorithms that utilize level-3 BLAS operations can easily be modified to incorporate TensorCore bFMAs, and some block algorithms operate more robustly in mixed precision than non-block techniques in certain regimes.
\section{Background: Build up to rounding error analysis for inner products}\label{sec:background}
In this section, we introduce the basic motivations and tools for mixed precision rounding error analysis needed for the {\it QR factorization}.
A matrix $\bb{A} \in \R^{m \times n}$ for $m\geq n$ can be written as
\[\bb{A} = \bb{QR}=\begin{bmatrix}\bb{Q}_{1} & \bb{Q}_2\end{bmatrix} \begin{bmatrix}\bb{R}_{1} \\ \bb{0}_{m-n \times n}\end{bmatrix} = \bb{Q}_{1}\bb{R}_{1},\vspace{-0.3cm}
\]
where an orthogonal $\bb{Q}\in\R^{m\times m}$ and an upper trapezoidal $\bb{R}$ form a \emph{full} QR factorization, and $\bb{Q}_{1}\in\R^{m\times n},\bb{R}_{1}\in\R^{n\times n}$ form a \emph{thin} QR factorization.
If $\bb{A}$ is full rank then the columns of $\bb{Q}_{1}$ are orthonormal
and $\bb{R}_1$ is upper triangular.
In many applications, only a \emph{thin} decomposition is needed as it produces an orthonormal basis for the column space of $\bb{A}$ while requiring less computation and storage.
Suppose that $\hat{y}$ is the result of implementing an algorithm that approximates $y=f(x)$.
The forward error is $|\hat{y}-y|$, and the backward error is defined as $|\Delta x|$ or $\min |\Delta x|$ that satisfies $\hat{y}=f(x+\Delta x)$. 
We use the first definition of backward error for the remainder of this paper, which for the QR factorization is: $\|\bb{A}-\hat{\bb{Q}}\hat{\bb{R}}\|$.
While important definitions are stated explicitly in the text, Table~\ref{table:notation} serves to establish basic notation.
\begin{table}[H]
	\centering
	\begin{tabular}{|m{2.cm}|m{10.6cm}|c|}
		\hline
		Symbol & Definition & Section \\ \hline
		${\bb x}$, ${\bb A}$,$|\bb{x}|$ ,$|\bb{A}|$  & Vector, matrix, and absolute value of each component  & \ref{sec:background} \\
		$\|{\bf x}\|_p$, $\|\bb{A}\|_p$ & Vector, operator $p$-norms for $p=2$, and Frobenius norm when $p=F$. & \ref{sec:background}\\
		$\bb{x}[i], \bb{A}[i,j]$,$\;:$& $i^{th}$ element of $\bb{x}$, $i^{th}$ row and $j^{th}$ column element of $\bb{A}$, all indices& \ref{sec:background}\\
		$\bb{X}_{m\times n}$, $\bb{X}_{n}$ & $m$-by-$n$ or $n$-by-$n$ matrices for $\bb{X}$ in $\{\bb{0},\bb{I}\}$, $\bb{I}_{m\times n}=[\bb{I}_{n} \quad \bb{0}_{n \times (m-n)}]^{\top}$ &  \ref{sec:intro}\\
		$\hat{\bb{e}}_i$  & $i^{th}$ cardinal vector &  \ref{sec:intro}\\
		\hline
		$\bb{Q},\bb{R}$  & Factors resulting from Householder (HH) QR factorization algorithms  & \ref{sec:background}\\
		$\bb{P}_{\bb{v}}$, $\bb{P}_i$ & HH transformation corresponding to $\bb{v}$, $i^{th}$ HH transformation in HQR& \ref{sec:algo}\\
		$\bb{X}, \bb{W}, \bb{Y}$ & WY representation of succesive HH transformations, $\bb{X}=\bb{I}-\bb{W}\bb{Y}^{\top}$ & \\
		\hline
		$\fl(\bb{x})$, $\hat{\bb{x}}$ & Quantity $\bb{x}$ calculated from floating point operations & \ref{sec:background} \\
		$\mu$, $\eta$  & mantissa, exponent bits of a floating point number & \ref{sec:background} \\
		$b_q$, $t_q$, $u^{(q)}$ & base, precision, unit round-off for precision $q$,  $u^{(q)}:=\frac{1}{2}b_q^{1-t_q}$& \ref{sec:background}\\
		$\dd^{(q)}$ &Quantity bounded by: $|\dd^{(q)}| < u^{(q)}$ &  \ref{sec:background} \\
		$\gamma^{(q)}_{k}$,  $\tth^{(q)}_{k}$& $\frac{ku^{(q)}}{1-ku^{(q)}}$, Quantity bounded by: $|\tilde{\tth}^{(q)}_{k}|\leq\tilde{\gamma}^{(q)}_{k}$ &  \ref{sec:background} \\
		$\tilde{\gamma}^{(q)}_{k}$,  $\tilde{\tth}^{(q)}_{k}$& $\frac{cku^{(q)}}{1-cku^{(q)}}$ for small integer $c>0$, Quantity bounded by: $|\tth^{(q)}_{k}|\leq\gamma^{(q)}_{k}$ &  \ref{sec:background} \\
		\hline
	\end{tabular}
	\caption{Basic definitions and where they first appear.}
	\label{table:notation}
\end{table}
\subsection{Basic rounding error analysis of floating point operations}\label{sec:backgroundRE}
We use and analyze the IEEE 754 Standard floating point number systems, shown in \cref{table:ieee}.
Let $\F \subset \R$ denote the space of some floating point number system with base $b\in\mathbb{N}$, precision $t\in\mathbb{N}$, significand $\mu\in\mathbb{N}$, and exponent range $[\eta_{\text{min}}, \eta_{\text{max}}]\subset \mathbb{Z}$.
Then every element $y$ in $\F$ can be written as 
\begin{equation}
y = \pm \mu\times b^{\eta-t},
\label{eqn:FPbasic}
\end{equation} 
where $\mu$ is any integer in $[0,b^{t}-1]$ and $\eta$ is an integer in  $[\eta_{\text{min}}, \eta_{\text{max}}]$.
Although operations we use on $\R$ cannot be replicated exactly due to the finite cardinality of $\F$, we can still approximate the accuracy of analogous floating point operations (FLOPs).
We adopt the rounding error analysis tools described in \cite{Higham2002}, which allow a relatively simple framework for formulating error bounds for complex linear algebra operations. 
An analysis of FLOPs (see Theorem 2.2 \cite{Higham2002}) shows that the relative error is 
controlled by the unit round-off, $u:=\frac{1}{2}b^{1-t}$ in uniform precision settings. 
In mixed precision settings we denote the higher precision unit round-off with $u^{(h)}$ (h for high) and the lower precision unit round-off with $u^{(l)}$ (l for low).\par 
\vspace{-.3cm}
\begin{table}[H]
	\begin{tabular}{||l|c|c|c|c|c|c||} 
		\hline 
		Name & $b$ & $t$ & \# of exponent bits & $\eta_{\text{min}}$ & $\eta_{\text{max}}$ & unit round-off $u$ \\ \hline 
		fp16 (IEEE754 half)& 2 & 11 & 5 & -15 & 16  & {\tt 4.883e-04} \\ \hline 
		fp32 (IEEE754 single)& 2 & 24 & 8 & -127 & 128  & {\tt 5.960e-08} \\ \hline 
		fp64 (IEEE754 double)& 2 & 53 & 11 & -1023 & 1024 & {\tt 1.110e-16} \\ \hline 
	\end{tabular}
	\caption{IEEE754 formats and their primary attributes.} 
	\label{table:ieee}
\end{table}
\vspace{-.8cm}

Let `op' be any basic operation from the set OP $=\{+, -, \times, \div\}$ and let $x,y\in \R$.
The true value $(x\text{ op }y)$ lies in $\R$, and it is rounded using some conversion to a floating point number, $\fl(x\text{ op }y)$, admitting a rounding error. 
The IEEE 754 Standard requires \emph{correct rounding}, which rounds the exact solution $(x\text{ op }y)$ to the closest floating point number and, in case of a tie, to the floating point number that has a mantissa ending in an even number.
\emph{Correct rounding} gives us an assumption for the error model where a single basic floating point operation yields a relative error, $\dd$, bounded in the following sense:
\begin{equation}
\fl(x\text{ op }y) = (1 + \dd)(x\text{ op }y),\quad |\dd|\leq u, \quad \text{op}\in\{+, -, \times, \div\}. \label{eqn:singlefpe}
\end{equation}
We use \cref{eqn:singlefpe} as a building block in accumulating errors from successive FLOPs.
Successive operations introduce multiple rounding error terms, and keeping track of all errors is challenging.
Lemma \ref{lem:gamma} introduces a convenient and elegant bound that simplifies accumulation of rounding error. 
\begin{lemma}[Lemma 3.1 \cite{Higham2002}]
	\label{lem:gamma}
	Let $|\dd_i|<u$, $\rho_i =\pm 1$ for $i=1:k$, and $ku < 1$. Then, 
	\begin{equation}
	\prod_{i=1}^k (1+\dd_i)^{\rho_i} = 1 + \tth_{k},
	\qquad \mbox{where } |\tth_{k}|\leq \frac{ku}{1-ku}=:\gamma_{k}.
	\end{equation}
	$$\mbox{Additionally, we define $\tilde{\tth}_k$ that satisfies }|\tilde{\tth}_k| \leq \tilde{\gamma}_k,\mbox{ where } \tilde{\gamma}_{k} = \frac{cku}{1-cku} \mbox{ for a small integer, $c>0$.}$$
\end{lemma}
In other words, $\tth_{k}$ represents the accumulation of rounding errors from $k$ successive operations, and it is bounded by $\gamma_{k}$.
In more complicated routines shown in later sections, we use the tilde notation ($\tilde{\gamma}_k$) to permit only keeping track of the leading order error terms.
Applying this lemma to the computation of $x+y+z$, where $x,y,z\in\R$, results in
\begin{equation}
\fl(x+y+z) = (1+\dd')\left((1+\dd) (x+y) +z\right) = (1+\tth_{2})(x+y) + (1+\tth_{1})z, \label{eqn:FPbasic4}
\end{equation}
where $|\dd|,|\dd'|<u$. 
Since $|\tth_{1}| \leq \gamma_{1} < \gamma_{2}$, we can further simplify \cref{eqn:FPbasic4} to
\begin{equation}
\fl(x+y+z) = (1+\tth'_{2})(x+y+z), \quad \mbox{where} \quad |\tth'_{2}| \leq \gamma_{2}, \label{eqn:FBbasic5}
\end{equation}
at the cost of a slightly larger upper bound. 
Note that both $|\tth_2|,|\tth'_2|$ are bounded above by $\gamma_2$.
Typically, error bounds formed in the fashion of \cref{eqn:FBbasic5} are converted to relative errors in order to put the error magnitudes in perspective.
The relative error bound for our example is
\begin{equation*}
|(x+y+z) - \fl(x+y+z)|\leq \gamma_{2}|x+y+z|,\;\; x+y+z\neq 0.
\end{equation*}

Although Lemma~\ref{lem:gamma} requires $ku<1$, we actually need $ku <\frac{1}{2}$ to maintain a meaningful relative error bound as this assumption implies $\gamma_k < 1$ and guarantees a relative error below 100\%. 
Since higher precision types have smaller unit round-offs, they can tolerate more successive FLOPs than lower precision floating types before reaching $\gamma_m=1$.
For example, the IEEE types introduced in \cref{table:ieee} meet this requirement at $1/2=2^{10}u^{(\text{fp16})}=2^{23}u^{(\text{fp32})}=2^{52}u^{(\text{fp64})}$.
Thus, accumulated rounding errors in lower precision types can lead to an instability with fewer operations in comparison to higher precision types and prompts us to evaluate whether existing algorithms can be naively adapted for mixed precision arithmetic.
\subsection{Rounding Error Example for the Inner Product}\label{sec:backgroundIP}
We now consider computing the inner product of two vectors to clearly illustrate how this situation restricts rounding error analysis in fp16. 
An error bound for an inner product of $m$-length vectors is
\begin{equation}
|\bb{x}^{\top}\bb{y} - \fl(\bb{x}^{\top}\bb{y})| \leq \gamma_{m} |\bb{x}|^{\top}|\bb{y}|, \quad \bb{x},\bb{y}\in\R^{m} \label{eqn:DDerr}
\end{equation}
as shown in \cite{Higham2002}.
Since vectors of length $m$ accumulate rounding errors that are bounded by $\gamma_{m}$, dot products of vectors computed in fp16 already face a 100\% relative error bound when $m=1024$. \par

A simple numerical experiment shows that the standard deterministic error bound is too pessimistic and cannot be practically used to approximate rounding error for half precision arithmetic. 
In this experiment, we generated 2 million random fp16 vectors of length $1024$ from two random distributions: the standard normal distribution, $N(0,1)$, and the uniform distribution over $(0,1)$.
Half precision arithmetic was simulated by calling \cref{algo:simulate}, which was proven to be a faithful simulation in \cite{HighamPranesh2019b}, for every FLOP (multiplication and addition for the dot product).
The relative error in this experiment is formulated as the LHS in Equation \ref{eqn:DDerr} divided by $|\bb{x}|^{\top}|\bb{y}|$ and all operations outside of calculating $\fl(\bb{x}^{\top}\bb{y})$ are executed by casting up to fp64 and using fp64 arithmetic.
Table \ref{table:HPdoterr} shows some statistics from computing the relative error for simulated fp16 dot products.
\begin{table}[h]
	\vspace*{-0cm}
	\centering
	\begin{tabular}{||c|c|c|c||} 
		\hline
		Random Distribution & Average & \makecell{Stan. Dev.}& Maximum\\ \hline
		Standard normal &{\tt 1.621e-04} & {\tt 1.635e-04 } & {\tt 3.204e-03}\\ \hline
		Uniform $(0,1)$ & {\tt 6.904e-03}& {\tt 3.265e-03} & {\tt 2.447e-02}\\ \hline
	\end{tabular}
	\caption{Forward error statistics from experiment of dot products computed in simulated half precision.}
	\label{table:HPdoterr}
\end{table}

We see that the inner products of vectors sampled from the standard normal distribution have relative errors that do not deviate much from the unit round-off ($\cO$({\tt 1e-4})), whereas the vectors sampled from the uniform distribution tend to accumulate larger errors on average ($\cO$({\tt 1e-3})). 
Even so, the theoretical upper error bound of 100\% is far too pessimistic as the maximum relative error does not even meet 2\% in this experiment.
Recent work in developing probabilistic bounds on rounding errors of floating point operations (see \cite{Higham2019a,Ipsen2019}) have shown that the inner product relative backward error for the conditions used for this experiment is bounded by {\tt 5.466e-2} with probability 0.99.
\begin{algorithm2e}
	\DontPrintSemicolon 
	\KwIn{$\bb{x}^{(\text{fp16})}$, $\bb{y}^{(\text{fp16})}$, $f$ \hfill\textbf{Output: } $\bb{z}^{(\text{fp16})}=\fl_{\text{fp16}}(f(\bb{x}^{(\text{fp16})}, \bb{y}^{(\text{fp16})}))$}
	$[\bb{x}^{(\text{fp32})}, \bb{y}^{(\text{fp32})}] \gets$ {\tt castup}$([\bb{x}^{(\text{fp16})},\bb{y}^{(\text{fp16})}])$\tcp*{Convert input vars to fp32.}
	$\bb{z}^{(\text{fp32})} \gets \fl(f(\bb{x}^{(\text{fp32})}, \bb{y}^{(\text{fp32})}))$ \tcp*{Perform fp32 arithmetic.}
	$\bb{z}^{(\text{fp16})} \gets$ {\tt castdown}$(\bb{z}^{(\text{fp32})})$\tcp*{Convert result to fp16.}
	\Return $\bb{z}^{(\text{fp16})}$
	\caption{$\bb{z}^{(\text{fp16})} = {\tt simHalf}(f, \bb{x}^{(\text{fp16})}, \bb{y}^{(\text{fp16})})$. Given fp16 input variables $\bb{x},\bb{y}$, perform function $f\in$ OP$\cup \{{\tt dot\_product} \}$ in simulated fp16 arithmetic. }
	\label{algo:simulate}
\end{algorithm2e}

Most importantly, we need error analysis that allows flexibility in precision in order to better our understanding of the impact of rounding errors on computations done on emerging hardware (i.e. GPUs) that support mixed precision.
We start by introducing some additional rules from \cite{Higham2002} that build on \cref{lem:gamma} in \cref{lem:up}. 
These rules summarize how to accumulate errors represented by $\tth$'s and $\gamma$'s in a \emph{uniform precision} setting.

\begin{lemma}
\label{lem:up}
For any positive integer $k$, let $\tth_{k}$ denote a quantity bounded according to $|\tth_{k}|\leq \frac{k u }{1-ku} =:\gamma_{k}$. The following relations hold for positive integers $j,n$ and nonnegative integer $k$.
Arithmetic operations between bounded terms, $\tth_{k}$'s, are: 
\begin{equation}
(1+\tth_{k})(1+\tth_{j})=(1+\tth_{k+j})\qquad \mbox{and} \qquad\frac{1+\tth_{k}}{1+\tth_{j}} =
\begin{cases}
	1+\tth_{k+j},& j \leq k\\
	1+\tth_{k+2j},& j > k\\
\end{cases} .
\end{equation}
If $\rm{max}(j,k) u \leq \frac{1}{2}$ and $n \leq \frac{1}{uk}$, the operations on the bounds, $\gamma$'s, are:
	\begin{align*}
	\gamma_{k}\gamma_{j} \leq \gamma_{\rm{min}(k,j)}&,\qquad n\gamma_{k} \leq \gamma_{nk},\\
	\gamma_{k} + u \leq \gamma_{k+1}&,\qquad \gamma_{k}+\gamma_{j}+\gamma_{k}\gamma_{j} \leq \gamma_{k+j}.
	\end{align*}
Note that all the rules hold when replaced by $\tilde{\gamma}$'s, but result in looser bounds.
\end{lemma}

We define two mixed precision settings that we use in \cref{sec:mpanalysis}.
In \cref{sec:mp-3}, we present the block Fused Multiply-Add (bFMA) of NVIDIA's TensorCore (TC) technology, which computes matrix-matrix multiply and accumulate for $4$-by-$4$ blocks, and incorporate it into \cref{algo:blockHQR,algo:par_tsqr}.
Here, we introduce an ad hoc mixed precision setting (MP Setting) which we use in \cref{sec:mp-2}.
This is explicitly defined in \cref{assump:mp} and is a level-2 BLAS variant of the TC bFMA. 
Both mixed precision settings define how inner products are computed although the bFMA is only applicable to inner products within matrix products and uses fp16 and fp32 whereas our ad hoc mixed precision setting is applicable to all inner products with any two precision types.\par

Our analysis is concerned with accuracy and stability and leaves out timing results of various hardwares.
Note that TCs perform matrix-matrix multiply and accumulate up to 8 times faster than fp32, and up to 16 times faster than fp64 (see \cite{Markidis2018}).

The exact product in \cref{assump:mp} is the result of using full precision products when the low precision type is fp16 and the high precision type is fp32 as is in TC bFMAs.
As a quick proof, consider $x^{(\text{fp16})} = \pm\mu_x2^{\eta_x -11},y^{(\text{fp16})} = \pm\mu_y2^{\eta_y -11}$ where $\mu_x,\mu_y\in[0,2^{11}-1]$ and $\eta_x,\eta_y\in[-15,16]$, and note that the significand and exponent ranges for fp32 are $[0, 2^{24}-1]$ and $[-127,128]$.
Then the product in full precision is
\[x^{(\text{fp16})}y^{(\text{fp16})} = \pm\mu_x\mu_y 2^{\eta_x+\eta_y+2-24},\]
where  $\mu_x\mu_y \in[0,(2^{11}-1)^2] \subseteq [0,2^{24}-1]$ and $\eta_x+\eta_y +2\in[-28,34]\subseteq[-127,128]$.
Thus, when two fp16 numbers are multiplied and stored in fp32, there is no roundoff error, and the summation and the cast down operations are the only sources of rounding error in this inner product scheme if no underflow or overflow occurs at the final cast down step.
\begin{assump}
	\label{assump:mp}
	Let $l$ and $h$ each denote low and high precision types with unit round-off values $u^{(l)}$ and $u^{(h)}$, where $1 \gg u^{(l)} \gg u^{(h)} >0$ and $u^{(h)} \lesssim (u^{(l)})^2$.
	Consider an FMA operation for inner products that take vectors stored in precision $l$, compute products exactly, and sum the products in precision $h$. 
	Finally, the result is then cast back down to precision $l$.
\end{assump}
We now analyze the rounding error for the inner product scheme described in \cref{assump:mp} and hypothesize that the guaranteed accuracy for this mixed precision inner product should be better than that of the low precision inner product and worse than that of the high precision inner product.
Let $\bb{x},\bb{y}$ be $m$-length vectors stored in a low precision type ($\F_l^m$), $s_k$ be the exact $k^{th}$ partial sum, and $\hat{s}_k$ be $s_k$ computed with FLOPs.
Then the first three partial sums are,
\begin{align*}
\hat{s}_1 &= \fl (\bb{x}[1]\bb{y}[1]) = \bb{x}[1]\bb{y}[1],\quad \hat{s}_2 = \fl(\hat{s}_1 + \bb{x}[2]\bb{y}[2]) = \left(\bb{x}[1]\bb{y}[1]+ \bb{x}[2]\bb{y}[2]\right)(1+\dd_{1}^{(h)}),\\
\hat{s}_3 &= \fl(\hat{s}_2+\bb{x}[3]\bb{y}[3]) = \left[\left(\bb{x}[1]\bb{y}[1] + \bb{x}[2]\bb{y}[2]\right)(1+\dd_{1}^{(h)})  + \bb{x}[3]\bb{y}[3]\right](1+\dd_{2}^{(h)}).
\end{align*}
We see a pattern emerging. 
The error for an $m$-length vector dot product is then
\begin{equation}
\label{eqn:dperr_2}
\hat{s}_m = (\bb{x}[1]\bb{y}[1]+\bb{x}[2]\bb{y}[2])\prod_{k=1}^{m-1}(1+\dd_{k}^{(h)}) + \sum_{i=3}^m \bb{x}[i]\bb{y}[i]\left(\prod_{k=i-1}^{m-1}(1+\dd_{k}^{(h)})\right).
\end{equation}
Using Lemma \ref{lem:gamma}, we further simplify and form componentwise backward errors with
\begin{equation}
\fl(\bb{x}^{\top}\bb{y}) =(\bb{x}+\Delta\bb{x})^{\top}\bb{y} = \bb{x}^{\top}(\bb{y}+\Delta\bb{y})\quad\text{for }|\Delta \bb{x}| \leq \gamma^{(h)}_{m-1}|\bb{x}|,\;\; |\Delta \bb{y}|  \leq \gamma_{m-1}^{(h)}|\bb{y}|. \label{eqn:beforecd}
\end{equation}
Casting down to $\F_l$ without underflow or overflow results in backward errors, 
\begin{equation}
\text{\tt castdown}(\fl(\bb{x}^{\top}\bb{y})) = (\bb{x}+\Delta\bb{x}+\tilde{\Delta}\bb{x})^{\top}\bb{y} = \bb{x}^{\top}(\bb{y}+\Delta\bb{y}+\tilde{\Delta}\bb{y}), \label{eqn:aftercd}
\end{equation}
where $|\Delta\bb{x} + \tilde{\Delta} \bb{x}| \leq ((1+u^{(l)})(1+\gamma_{m-1}^{(h)})-1)|\bb{x}|$ and $|\Delta\bb{y}+\tilde{\Delta} \bb{y}| \leq ((1+u^{(l)})(1+\gamma_{m-1}^{(h)})-1)|\bb{y}|$.
Our hypothesis is indeed true since,
\[\gamma_m^{(h)}<u^{(l)}+\gamma_{m-1}^{(h)}+u^{(l)}\gamma_{m-1}^{(h)}<\gamma_{m}^{(l)},\]
where the lower and upper bounds are derived from the uniform precision error bound in \cref{eqn:DDerr}. 
\Cref{eqn:aftercd} shows us that the two larger error terms are from the higher precision summation, $\gamma_{m-1}^{(h)}$, and the cast down operation, $u^{(l)}$.
We can measure the impact of the cast down step relative to the length of the vector, $m$, and the disparity in the two precisions, $M_{l,h}:=u^{(l)}/u^{(h)}$, since these two factors determine which one of $u^{(l)}$ and $mu^{(h)}$ is the leading order term. 
We consider 3 cases.\\
\textbf{Case 1: ($m\ll M_{l,h}$)} The leading order term is $u^{(l)}$.
The mixed precision inner product has a smaller worst case error bound than the bound of the low precision inner product ($mu^{(l)}$).
On the other hand, $u^{(l)}$ is a larger upper bound than that of the high precision inner product ($mu^{(h)}=\frac{m}{M_{l,h}}u^{(l)}$).
It is likely that this factor of $M_{l,h}/m$ increase in the worst case error bound is unwanted.
\\
\textbf{Case 2: ($m = M_{l,h}$)}
Both terms are now leading order. 
This is still an improvement in comparison to the lower precision arithmetic as the error bound is reduced from $mu^{(l)}$ to $2u^{(l)}$.
Comparing this with the high precision inner product shows that the error bound has doubled from $mu^{(h)}$ to $2mu^{(h)}$.
\textbf{Case 3: ($m \gg M_{l,h}$)}
Now $\gamma_{m-1}^{(h)}$ is the leading order term. 
As in the above two cases, this is an improvement in the context of the low precision accuracy since the error has been reduced from $\gamma_m^{(l)}$ to $\gamma_{m/M_{l,h}}^{(l)}\equiv \gamma_m^{(h)}$. 
Since $u^{(l)} = M_{l,h}u^{(h)} \ll mu^{(h)}$, the mixed precision error bound has the same \emph{order} as the error bound from carrying the computation out in the higher precision. 
Therefore, we can expect about the same level of accuracy.\par
Finally, we present alternative representations of the error bound in \cref{eqn:aftercd},
\begin{align*}
(1+u^{(l)})(1+\gamma_{m-1}^{(h)})-1 &\leq \gamma_{M_{l,h}+m-1}^{(h)}=\gamma_{1+(m-1)/M_{l,h}}^{(l)}, \;\; M_{l,h} = u^{(l)}/u^{(h)},\\
(1+u^{(l)})(1+\gamma_{m-1}^{(h)})-1 &\leq  u^{(l)} + \gamma_{m-1}^{(h)} + \min\{u^{(l)}, \gamma_{m-1}^{(h)}\},\;\; \gamma_{m-1}^{(h)} < 1,
\end{align*}
where the rules from \cref{lem:up} were directly applied.
Both alternative bounds are only slightly larger than the original bound shown on the LHS and remain in the same order.
The first is useful when comparing against the low or the high precision, whereas the second keeps track of the error bounds in both precisions.
We summarize these ways of combining $\gamma$ terms of different precisions in \cref{lem:mp},
\begin{lemma}\label{lem:mp}
	For any nonnegative integers $k_l$, $k_h$ and some precision $q$ defined with respect to the unit round-off, $u^{(q)}$, define $\gamma^{(q)}_{k} := \frac{k u^{(q)} }{1-ku^{(q)}}$.
	Consider a low precision and a high precision where $1 \gg u^{(l)} \gg u^{(h)}>0$, and $k_l$, $k_h$ that satisfy $\max\{\gamma^{(h)}_{k_{h}},\gamma^{(l)}_{k_{l}}\} < 1/2$.
	Then the following rules help us accumulate $\gamma$'s of different precisions,
	\begin{align}
	\gamma^{(h)}_{k_{h}}\gamma^{(l)}_{k_{l}} &\leq \min\{\gamma^{(h)}_{k_{h}},\gamma^{(l)}_{k_{l}} \},\\ 
	(1+\tilde{\gamma}_{k_l}^{(l)})(1+\tilde{\gamma}_{k_h}^{(h)}) -1 &= \tilde{\gamma}_{k_l}^{(l)}+\tilde{\gamma}_{k_h}^{(h)}. \label{eqn:mpgamma}
	\end{align}
\end{lemma}
Note that \cref{eqn:mpgamma} drops the term $\tilde{\gamma}_{k_l}^{(l)}\tilde{\gamma}_{k_h}^{(h)}$ since both $\tilde{\gamma}_{k_l}^{(l)}$ and $\tilde{\gamma}_{k_h}^{(h)}$ are larger than their product and this product can be swept into the small integer $c > 0$ implicitly included in the tilde notation.
Using these two mixed precision settings (TC bFMA and \cref{assump:mp}) in HQR algorithms results in casting down to the low precision at different parts of the algorithms.
In general, error bounds in the fashion of \cref{eqn:beforecd} correspond to rounding errors prior to cast down operations, and cast down operations introduce an additional error term as in and error bounds similar to \cref{eqn:aftercd}.\par

We have demonstrated a need for rounding error analysis that is accurate for mixed precision procedures and analyzed the inner product in an ad hoc mixed precision setting that mimics the TensorCore bFMA.
We will use this to analyze various HH QR factorization algorithms.
Algorithms and the general framework for the standard rounding error analysis for these algorithms are introduced in \cref{sec:algo}, and both are modified to meet different mixed precision assumptions in \cref{sec:mpanalysis}.
\section{Algorithms and existing round-off error analyses}\label{sec:algo}
We introduce the Householder QR factorization algorithm (HQR) in \cref{sec:HQR} and two block variants that use HQR within the block in \cref{sec:BQR,sec:TSQR}. 
The blocked HQR (BQR) in \cref{sec:BQR} partitions the columns of the target matrix and is a well-known algorithm that uses the WY representation of \cite{Bischof1987} that utilizes mainly level-3 BLAS operations.
In contrast, the Tall-and-Skinny QR (TSQR) in \cref{sec:TSQR} partitions the rows and takes a communication-avoiding divide-and-conquer approach that can be easily parallelized (see \cite{Demmel2007}).
We present the standard rounding error analysis of these algorithms (see \cite{Higham2002,Mori2012}) which will be tweaked for various mixed precision assumptions in \cref{sec:mpanalysis}.
\subsection{Householder QR (HQR)}\label{sec:HQR}
The HQR algorithm uses HH transformations to zero out elements below the diagonal of a matrix (see \cite{Householder1958}). 
We present this as zeroing out all but the first element of some vector, $\bb{x}\in\R^m$.
\begin{lemma}
	Given vector $\bb{x}\in\R^{m}$, there exist an HH vector , $\bb{v}$, and an HH constant, $\beta$, that define the HH transformation matrix, $\bb{P}_{\bb{v}}:=\bb{I}_{m} - \beta \bb{v}\bb{v}^{\top}$, such that $\bb{P}_{\bb{v}}$ zeroes out $\bb{x}$ below the first element. 
	The HH vector and constant are defined via
	\begin{equation}
	\sigma = -\rm{sign}(\bb{x}[1])\|\bb{x}\|_2, \quad  \bb{v} = \bb{x} -\sigma \hat{\bb{e}}_{1},\mbox{ and } \beta =\frac{2}{\bb{v}^{\top}\bb{v}}=-\frac{1}{\sigma\bb{v}[1]}.
	\label{eqn:HH} 
	\vspace*{-.3cm}
	\end{equation}
	The transformed vector, $\bb{P_vx}=\sigma\hat{\bb{e}_1}$, has the same 2-norm as $\bb{x}$ since $\bb{P}_{\bb{v}}=\bb{P}_{\bb{v}}^{\top}=\bb{P}_{\bb{v}}^{-1}$.
	\label{lem:hhvec}
\end{lemma}
\subsubsection{HQR: Algorithm}
Given $\bb{A}\in\R^{m\times n}$ and Lemma \ref{lem:hhvec}, HQR is done by repeating the following processes until only an upper triangle matrix remains.
For $i = 1, 2, \cdots, n,$
\begin{enumerate}[Step 1)]
	\item Compute $\bb{v}$ and $\beta$ that zeroes out the $i^{th}$ column of $\bb{A}$ beneath $a_{ii}$ (see \cref{algo:hh_v2}), and
	\item Apply $\bb{P}_{\bb{v}}$ to the bottom right partition, $\bb{A}[i:m, i:n]$ (lines 4-6 of \cref{algo:hhQR}).
\end{enumerate}

Consider the following $4$-by-$3$ matrix example adapted from \cite{Higham2002}. 
Let $\bb{P}_{i}$ represent the $i^{th}$ HH transformation of this algorithm. 
\[\scriptstyle\bb{A} = \left[ \begin{array}{ccc}
\times & \times & \times \\
\times & \times & \times \\
\times & \times & \times \\
\times & \times & \times
\end{array}
\right]\xrightarrow{\bb{P}_{1}\bb{A}}\left[ \begin{array}{c|cc}
\times & \times & \times \\ \hline
0 & \times & \times \\
0 & \times & \times \\
0 & \times & \times
\end{array}
\right]
\xrightarrow{\bb{P}_{2}\bb{P}_{1}\bb{A}}\left[
\begin{array}{cc|c}
\times & \times & \times \\
0 & \times & \times \\ \hline
0 & 0 & \times \\
0 & 0 & \times 
\end{array} \right]
\xrightarrow{\bb{P}_{3}\bb{P}_{2}\bb{P}_{1}\bb{A}} \left[ \begin{array}{ccc}
\times & \times & \times \\
0 & \times & \times \\
0 & 0 & \times \\
0 & 0 & 0 
\end{array}\right]\]
The resulting matrix is the $\bb{R}$ factor, $\bb{R}:= \bb{P}_{3}\bb{P}_{2}\bb{P}_{1}\bb{A}$, and the $\bb{Q}$ factor for a full QR factorization is $\bb{Q}:=\bb{P}_{1}\bb{P}_{2}\bb{P}_{3}$ since $\bb{P}_{i}$'s are symmetric.
The thin factors for a general matrix $\bb{A}\in\R^{m\times n}$ are
\begin{equation}
\bb{Q}_{\text{thin}} = \bb{P}_{1} \cdots \bb{P}_{n}\bb{I}_{m\times n}\quad \text{and} \quad \bb{R}_{\text{thin}} = \bb{I}_{m\times n}^{\top}\bb{P}_{n}\cdots \bb{P}_{1}\bb{A}.
\end{equation}

\begin{algorithm2e}[H]
	\DontPrintSemicolon 
	\KwIn{$\bb{x}$ \hfill \textbf{Output: }$\bb{v}$, $\sigma$, and $\beta $}
	$\bb{v}\gets$ {\tt copy}($\bb{x}$)\\
	$\sigma \gets -\rm{sign}(\bb{x}[1])\|\bb{x}\|_2$\\
	$\bb{v}[1] \gets \bb{x}[1]-\sigma$ \\
	$\beta \gets -\frac{\bb{v}[1]}{\sigma}$\\
	${\bf v} \leftarrow {\bf v} / {\bf v}[1]$\\
	\Return $\beta$, $\bb{v}$, $\sigma$
	\caption{$\beta$, $\bb{v}$, $\sigma = {\tt hhvec}(\bb{x})$. Given a vector $\bb{x}\in\R^m$, return $\bb{v}\in\R^m$ and $\beta,\sigma\in\R$ that satisfy $(I-\beta \bb{v}\bb{v}^{\top})\bb{x} =\sigma\hat{\bb{e}}_{1}$ and $\bb{v}[1]=1$ (see \cite{LAPACK, Higham2002}).}
	\label{algo:hh_v2}
\end{algorithm2e}

\begin{algorithm2e}
	\DontPrintSemicolon
	\KwIn{$\bb{A}$ \hfill \textbf{Output: }$\bb{V}$,$\bm{\beta}$, $\bb{R}$}
	Initialize $\bb{V} \gets \bb{0}_{m\times n}$, $\bm{\beta} \gets \bb{0}_m$ \\
	
	\For{$i=1 : n$}{
		$\bb{v}, \beta, \sigma \gets$ {\tt hhvec}$(\bb{A}[i:\mathrm{end}, i])$ \tcc*{\Cref{algo:hh_v2}}
		$\bb{V}[i:\mathrm{end},i]$, $\bm{\beta}_i$,  $\bb{A}[i,i] \gets \bb{v}, \beta, \sigma$\\
		$\bb{A}[i+1:\mathrm{end}, i]\gets \mathrm{zeros}(m-i)$\\
		$\bb{A}[i:\mathrm{end}, i+1:\mathrm{end}]\gets \bb{A}[i:\mathrm{end}, i+1:\mathrm{end}] - \beta \bb{v} \bb{v}^{\top}\bb{A}[i:\mathrm{end}, i+1:\mathrm{end}]$
	}
	\Return $\bb{V}$, $\bm{\beta}$, $\bb{A}[1:n, 1:n]$
	\caption{$\bb{V}$, $\bm{\beta}$, $\bb{R} =$ {\tt HQR2}$(\bb{A})$. A Level-2 BLAS implementation of HQR. Given a matrix $\bb{A}\in\R^{m\times n}$ where $m\geq n$, return matrix $\bb{V}\in\R^{m\times n}$, vector $\bm{\beta}\in\R^{n}$, and upper triangular matrix $\bb{R}$. The orthogonal factor $\bb{Q}$ can be generated from $\bb{V}$ and $\bm{\beta}$.}
	\label{algo:hhQR}
\end{algorithm2e}
\subsubsection{HQR: Rounding Error Analysis}
Now we present an error analysis for \cref{algo:hhQR} by keeping track of the different operations of \cref{algo:hh_v2} and \cref{algo:hhQR}.
We follow the analysis of \cite{Higham2002} and modify it for the variant where $\bb{v}[1]$ is set to $1$.
The goal of this section is to present the basic steps of the standard error analysis for HQR so that we modify them easily in \cref{sec:mpanalysis} for different mixed precision settings.
\paragraph{Calculating the $i^{th}$ HH vector and constant} 
In \cref{algo:hhQR}, we compute the HH vector and constant by using \cref{algo:hh_v2} to $\bb{A}[i:m,i]$.
For now, consider zeroing out any vector $\bb{x}\in\R^m$ below its first component with an HH transformation.
We first calculate $\sigma$ as is implemented in line 2 of \cref{algo:hh_v2}.
\begin{equation}
\label{eqn:sigma}
\fl(\sigma) = \hat{\sigma} = \rm{fl}(-\rm{sign}(\bb{x}[1])\|\bb{x}\|_2) = \sigma + \Delta \sigma,\quad |\Delta\sigma| \leq \gamma_{m+1}|\sigma|.
\end{equation}
Note that the backward error incurred here accounts for an inner product of a vector in $\R^{m}$ with itself and a square root operation to get the 2-norm. 
Let $\bb{v}'[1]\equiv \bb{x}[i]-\sigma$, the penultimate value $\bb{v}[1]$ held. 
The subtraction adds a single additional rounding error via
\begin{equation}
	\fl(\bb{v}'[1]) =\bb{v}'[1] + \Delta \bb{v}'[1] = (1+\dd) (\bb{x}[i]-\sigma-\Delta\sigma)= (1+\tth_{m+2})\bb{v}'[1]
\end{equation}
where the last equality is granted because the sign of $\sigma$ is chosen to prevent cancellation.
Since \cref{algo:hh_v2} normalizes the HH vector so that its first component is $1$, the remaining components of $\bb{v}$ are divided by $\fl(\tilde{\bb{v}}_1)$ incurring another single rounding error.
As a result, the components of $\bb{v}$ computed with FLOPs have error $\fl(\bb{v}[j])	= \bb{v}[j] + \Delta \bb{v}[j]$ where 
\begin{equation}
|\Delta \bb{v}[j]|\leq \gamma_{1+2(m+2)}|\bb{v}[j]| =\tilde{\gamma}_{m}|\bb{v}[j]|\quad j=2:m-i+1,\label{eqn:vbound}
\end{equation}
and $|\Delta {\bf v}[1]| = 0$.
Since $1+2(m+2) = \cO(m)$, we have swept that minor difference into the constant defined in the $\tilde{\gamma}$ notation.
Next, we consider the HH constant, $\beta$, as is computed in line 4 of \cref{algo:hh_v2}.
\begin{align}
\hat{\beta} = \fl\left(-\bb{v}'[1]/\hat{\sigma}\right) &=-(1+\dd)\frac{\bb{v}'[1]+\Delta \bb{v}'[1]}{\sigma + \Delta\sigma} = \frac{(1+\dd)(1+\tth_{m+2})}{(1+\tth_{m+1})}\beta \label{eqn:beta}\\
&= (1+\tth_{2m+4})\beta= \beta + \Delta \beta,\text{ where } |\Delta\beta| \leq \tilde{\gamma}_{m} \beta\label{eqn:beta3}.
\end{align}
We have shown \cref{eqn:beta} to keep our analysis simple in \cref{sec:mpanalysis} and \cref{eqn:beta3} to show that the error incurred from calculating $\|\bb{x}\|_2$ accounts for the vast majority of the rounding error so far.
In iteration $i$, we replace $\bb{x}$ with $\bb{A}[i:m,i]\in\R^{m-i+1}$ and the $i^{th}$ HH constant and vector ($\hat{\beta}_i$,$\bb{v}_i$) both have errors bounded by $\tilde{\gamma}_{m-i+1}$.
\paragraph{Applying a Single HH Transformation}
Now we consider lines 4-6 of \cref{algo:hhQR}.
In iteration $i$, we set $\bb{A}[i+1:m,:]$ to zero and replace $\bb{A}[i,i]$ with $\sigma$ computed from \cref{algo:hh_v2}.
Therefore, we now need to calculate the errors for applying an HH transformation to the remaining columns, $\bb{A}[i:m, i+1:n]$ with the computed HH vector and constant.
This is the most crucial building block of the rounding error analysis for any variant of HQR because the $\bb{R}$ factor is formed by applying the HH transformations to $\bb{A}$ and the $\bb{Q}$ factor is formed by applying them in reverse order to the identity.
Both of the blocked versions in \cref{sec:BQR} and \cref{sec:TSQR} also require slightly different but efficient implementations of this step.
For example, BQR in \cref{algo:blockHQR} uses level-3 BLAS operations to apply multiple HH transformations at once whereas the variant of HQR in \cref{algo:hhQR} can only use level-2 BLAS operations to apply HH transformations.\par

A HH transformation is applied through a series of inner and outer products, since HH matrices are rank-1 updates of the identity. 
That is, computing  $\bb{P}_{\bb{v}}\bb{x}$ for any $\bb{x}\in\R^{m}$ is as simple as computing 
\begin{equation}
	\bb{y}:=\bb{P}_{\bb{v}}\bb{x} = \bb{x} - (\beta \bb{v}^{\top}\bb{x})\bb{v}.\label{eqn:effH}
\end{equation}
Let us assume that $\bb{x}$ is an exact vector and there were errors incurred in forming $\bb{v}$ and $\beta$. 
The errors incurred from computing $\bb{v}$ and $\beta$ need to be included in addition to the new rounding errors accumulating from the action of applying $\bb{P}_{\bb{v}}$ to a column.
In practice, $\bb{x}$ is any column in $\bb{A}^{(i-1)}[i+1:m, i+1:n]$, where the superscript $(i-1)$ indicates that this submatrix of $\bb{A}$ has already been transformed by $i-1$ HH transformations that zeroed out components below $\bb{A}[j,j]$ for $j = 1:i-1$.
We show the error for forming $\hat{\bb{w}}$ where $\bb{w}:=\beta(\bb{v}^{\top}\bb{x})\bb{v}$ and $\bb{v},\bb{x}\in\R^{m}$,
\begin{equation*}
\hat{\bb{w}} =\fl(\hat{\beta}\;\fl(\hat{\bb{v}}^{\top}\bb{x})\hat{\bb{v}})=(1+\tth_{m})(1+\dd)(1+\dd')(\beta+\Delta\beta)\left((\bb{v}+\Delta \bb{v})^{\top}\bb{x}\right)(\bb{v}+\Delta \bb{v}),
\end{equation*}
where $\tth_{m}$ is from computing the inner product $(\hat{\bb{v}}^{\top}\bb{x})$, and $\dd$ and $\dd'$ are from multiplying $\beta$, $\fl(\hat{\bb{v}}^{\top}\bb{x})$, and $\bb{\hat{v}}$.
The forward error is
$\hat{\bb{w}} = \bb{w} + \Delta \bb{w}$, where $|\Delta \bb{w}| \leq \tilde{\gamma}_m|\beta|\left(|\bb{v}|^{\top}|\bb{x}|\right)|\bb{v}|.$
Subtracting $\hat{\bb{w}}$ from $\bb{x}$ yields the HH transformation with forward error,
\begin{equation}
\fl(\hat{\bb{P_v}}\bb{x}) = \fl(\bb{x}-\bb{\hat{w}}) = (1+\dd)(\bb{x}-\bb{w}-\Delta \bb{w}) = \bb{y} + \Delta \bb{y} = (\bb{P_v} + \Delta \bb{P_v})\bb{x},\label{eqn:applyP}
\end{equation}
where $|\Delta \bb{y}| \leq u|\bb{x}| + \tilde{\gamma}_{m} |\beta||\bb{v}||\bb{v}|^{\top}|\bb{x}|$.
Using $\sqrt{2/\beta} = \|\bb{v}\|_2$, we form a normwise bound,
\begin{equation}
\|\Delta \bb{y}\|_2 \leq \tilde{\gamma}_{m}\|\bb{x}\|_2. \label{eqn:19.2c}
\end{equation}
Since $\Delta \bb{P_v}[i,j] = \frac{1}{\|\bb{x}\|_2^2}\Delta \bb{y}[i]\bb{x}[j]$, we can compute its Frobenius norm,
	\begin{equation}
	\|\Delta \bb{P_v}\|_F
	= \left(\sum_{i=1}^m\sum_{j=1}^m\left(\frac{1}{\|\bb{x}\|_2^2}\Delta \bb{y}[i]\bb{x}[j]\right)^2\right)^{1/2}
	=  \frac{\|\Delta \bb{y}\|_2}{\|\bb{x}\|_2} \leq \tilde{\gamma}_{m}\label{eqn:outer},
	\end{equation}
where the last inequality is a direct application of \cref{eqn:19.2c}.
\paragraph{Applying many successive HH transformations}
Consider applying a sequence of transformations in the set $\{\bb{P}_{i}\}_{i=1}^r\subset\R^{m\times m}$ to $\bb{x}\in\R^m$, where $\bb{P}_{i}$'s are all HH transformations computed with $\tilde{\bb{v}}_i$'s and $\hat{\beta_i}$'s.
This is directly applicable to HQR as $\bb{Q}=\bb{P}_{1}\cdots\bb{P}_{n}\bb{I}$ and $\bb{R} = \bb{Q}^{\top}\bb{A} = \bb{P}_{n}\cdots\bb{P}_{1}\bb{A}$.
\Cref{lem:3.7} is very useful for any sequence of transformations, where each transformation has a known bound.
We will invoke this lemma to prove \cref{lem:19.3}, and use it in future sections for other consecutive transformations.
\begin{lemma}\label{lem:3.7}
	If $\bb{X}_{j} + \Delta \bb{X}_{j} \in\R^{m\times m}$ satisfies $\|\Delta \bb{X}_{j}\|_F\leq \tau_j \|\bb{X}_{j}\|_2$ for all $j=1,\cdots,r$, then $$\left|\left|\prod_{j=1}^r (\bb{X}_{j} + \Delta \bb{X}_{j})- \prod_{j=1}^r\bb{X}_{j} \right|\right|_F\leq\left(-1+\prod_{j=1}^r(1+\tau_j)\right)\prod_{j=1}^r\|\bb{X}_{j}\|_2.$$
\end{lemma}

\begin{lemma}\label{lem:19.3}
	Consider applying a sequence of transformations $\bb{Q}=\bb{P}_{r}\cdots\bb{P}_{2}\bb{P}_1$ onto vector $\bb{x}\in\R^m$ to form $\hat{\bb{y}} =\fl(\hat{\bb{P}}_{r}\cdots\hat{\bb{P}}_{2}\hat{\bb{P}}_{1}\bb{x}),$
	where $\hat{\bb{P}}_{k}$'s are HH transformations constructed from $\hat{\beta}_k$ and $\hat{\bb{v}}_{k}$.
	These HH vectors and constants are computed via \cref{algo:hh_v2} and the rounding errors are bounded by \cref{eqn:beta3,eqn:vbound}.
	If each transformation is computed via \cref{eqn:effH}, then
	\begin{align}
	\hat{\bb{y}} &= \bb{Q} (\bb{x} +\Delta \bb{x}) = (\bb{Q} + \Delta \bb{Q})\bb{x} = \hat{\bb{Q}}\bb{x},\\
	\|\Delta \bb{y}\|_2 &\leq r\tilde{\gamma}_m\|\bb{x}\|_2,\;\; \|\Delta \bb{Q}\|_F\leq r\tilde{\gamma}_m .\label{eqn:19.3}
	\end{align}
\end{lemma}
\begin{proof}
	Applying \cref{lem:3.7} directly to $\bb{Q}$ yields
	\[
		\|\Delta\bb{Q}\|_F = \left|\left|\prod_{j=1}^r (\bb{P}_{j} + \Delta \bb{P}_{j})- \prod_{j=1}^r\bb{P}_{j} \right|\right|_F\leq\left(-1+\prod_{j-1}^r(1+\tilde{\gamma}_{m-j+1})^r\right)\prod_{j=1}^r\|\bb{P}_{j}\|_2 \leq -1+(1+\tilde{\gamma}_m)^r,\]
	since $\bb{P}_{j}$'s are orthogonal and have unit 2-norm and $m-j+1 \leq m$.
	While we omit the details here, we can show that $(1+\tilde{\gamma}_m)^r-1 \leq r\tilde{\gamma}_m$ using the argument from \cref{lem:gamma} if $r\tilde{\gamma}_m \leq 1/2$.
\end{proof}
In this error analysis, the prevailing bound for errors at various stages of forming and applying an HH transformation is $\tilde{\gamma}_{m}$ where $m$ corresponds to the dimension of the transformed vectors.
In \cref{lem:19.3}, a factor of $r$ is introduced for applying $r$ HH transformations to form the term $r\tilde{\gamma}_m \approx rmu$. 
Therefore, we can expect that the columnwise norm error for a thin QR factorization should be $\cO(mnu)$ for a full rank matrix.
In \cref{thm:feHQR}, we formalize this by applying \cref{lem:19.3} directly and also show a conversion of columnwise bounds to a matrix norm bound,
\begin{equation*}
	\|\Delta \bb{R} \|_F = \left(\sum_{i=1}^n \|\Delta \bb{R}[:,i]\|_2^2\right)^{1/2} \leq \left(\sum_{i=1}^n n^2\tilde{\gamma}_{m}^2 \|\bb{A}[:,i]\|_2^2\right)^{1/2} = n\tilde{\gamma}_{m} \|\bb{A}\|_F.
\end{equation*}
We gather these results into \cref{thm:feHQR}.
\begin{theorem}
	\label{thm:feHQR}
	Let $\bb{A}\in\R^{m\times n}$ with $m\geq n$ have full rank, $n$. 
	Let $\hat{\bb{Q}}\in\R^{m\times n}$ and $\hat{\bb{R}}\in\R^{n\times n}$ be the thin QR factors of $\bb{A}$ obtained via \cref{algo:hhQR}.
	Then,
	\begin{align*}
	\hat{\bb{R}} &= \bb{R} + \Delta \bb{R} = \fl(\hat{\bb{P}}_n\cdots\hat{\bb{P}}_1 \bb{A}),\;\; \|\Delta \bb{R}[:,j]\|_2\leq n\tilde{\gamma}_{m} \|\bb{A}[:,j]\|_2,\;\; \|\Delta \bb{R}\|_F\leq n\tilde{\gamma}_{m} \|\bb{A}\|_F\\
	\hat{\bb{Q}} &= \bb{Q} + \Delta \bb{Q} = \fl(\hat{\bb{P}}_1\cdots\hat{\bb{P}}_n \bb{I}),\;\; \|\Delta \bb{Q}[:,j]\|_2\leq n\tilde{\gamma}_{m},\;\; \|\Delta \bb{Q}\|_F \leq n^{3/2} \tilde{\gamma}_{m}.
	\end{align*}
\end{theorem}
In future sections, we show the forward error columnwise bounds for each factor which can be easily converted to matrix norm bounds.
The numerical experiments in \cref{sec:NE} measure backward errors with $\|\hat{\bb{Q}}\hat{\bb{R}}-\bb{A}\|_F$ and the orthogonality of the $\bb{Q}$ factor with $\|\hat{\bb{Q}}^{\top}\hat{\bb{Q}}-\bb{I}\|_2$.
The content of this section shows the standard rounding error analysis in \cite{Higham2002} where some important stages are summarized in \cref{eqn:beta3,eqn:vbound,eqn:19.3}, which we will modify to different mixed precision settings in \cref{sec:mpanalysis}. 
These quantities account for various forward and backward errors formed in computing essential components of HQR, namely the HH constant and vector, as well as normwise errors of the action of applying HH transformations.
In the next sections, we present blocked variants of HQR that use \cref{algo:hhQR}.
\subsection{Block HQR with partitioned columns (BQR)}\label{sec:BQR}
We refer to the blocked variant of HQR where the columns are partitioned as BQR. 
Note that this section relies on the WY representation described in \cite{Bischof1987} instead of the storage-efficient version of \cite{Schreiber1989}, even though both are known to be just as numerically stable as HQR.
\subsubsection{The WY Representation}
A convenient matrix representation that accumulates $r$ HH reflectors is known as the WY representation (see \cite{Bischof1987,golub2013matrix}).
\Cref{lem:WY} shows how to update a rank-$j$ update of the identity, $\bb{Q}^{(j)}$, with an HH transformation, $\bb{P}$, to produce a rank-$(j+1)$ update of the identity, $\bb{Q}^{(j+1)}$. 
With the correct initialization of $\bb{W}$ and $\bb{Y}$, we can build the WY representation of successive HH transformations as shown in \Cref{algo:buildWY}. 
This algorithm assumes that the HH vectors, $\bb{V}$, and constants, $\bm{\beta}$, have already been computed.
Since the $\bb{Y}$ factor is exactly $\bb{V}$, we only need to compute the $\bb{W}$ factor.
\begin{lemma}\label{lem:WY}
	Suppose $\bb{X}^{(j)}=\bb{I}-\bb{W}^{(j)}\bb{Y}^{(j)\top}\in\R^{m\times m}$ is an orthogonal matrix with $\bb{W}^{(j)},\bb{Y}^{(j)}\in\R^{m\times j}$.
	Let us define $\bb{P}=\bb{I}-\beta\bb{vv}^{\top}$ for some $\bb{v}\in\R^m$ and let $\bb{z}^{(j+1)}=\beta\bb{X}^{(j)}\bb{v}$.
	Then, \[\bb{X}^{(j+1)} = \bb{X}^{(j)}\bb{P} = \bb{I} - \bb{W}^{(j+1)}\bb{Y}^{(j+1)\top}, \]where $ \bb{W}^{(j+1)} =[\bb{W}^{(j)}|\bb{z}]$ and $ \bb{Y}^{(j+1)}=[\bb{Y}^{(j)}|\bb{v}]$ are each $m$-by-$(j+1)$. 
\end{lemma}
\begin{algorithm2e}
	\DontPrintSemicolon 
	\KwIn{$\bb{V}\in\R^{m \times r}$, $\bm{\beta}\in\R^{r}$ where $m > r$.}	
	\KwOut{$\bb{W}$} 
	Initialize: $\bb{W}:=\bm{\beta}_1\bb{V}[:,1]$.\tcc*{$\bb{Y}$ is $\bb{V}$.}
	\For{$j=2:r$}{
		$\bb{z}\gets \bm{\beta}_j \left[\bb{V}[:,j] - \bb{W}\left(\bb{V}[:,1:j-1]^{\top}\bb{V}[:,j]\right)\right]$\\
		$\bb{W} \gets [\bb{W}\quad \bb{z}]$ \tcc*{Update $\bb{W}$ to an $m$-by-$j$ matrix.}
	}
	\Return $\bb{W}$
	\caption{$\bb{W},\bb{Y}\gets {\tt buildWY}(V, \bm{\beta})$: Given a set of Householder vectors $\{\bb{V}[:,i]\}_{i=1}^r$ and their corresponding constants $\{\bm{\beta}_i\}_{i=1}^r$, form the final $\bb{W}$ and $\bb{Y}$ factors of the WY representation of $\bb{P}_1\cdots \bb{P}_r$, where $\bb{P}_i := \bb{I}_m - \bm{\beta}_i\bb{v}_i\bb{v}_i^{\top}$}
	\label{algo:buildWY}
\end{algorithm2e}

In HQR, $\bb{A}$ is transformed into an upper triangular matrix $\bb{R}$ by identifying an HH transformation that zeroes out a column below the diagonal, then applying that HH transformation to the bottom right partition. 
For example, the $k^{th}$ HH transformation finds an $m-k+1$ sized HH transformation that zeroes out column $k$ below the diagonal and then applies it to the $(m-k+1)$-by-$(n-k)$ partition of the matrix $\bb{A}[k:m,k+1:n]$.
Since the $(k+1)^{th}$ column is transformed by the $k^{th}$ HH transformation, this algorithm must be executed serially as shown in \cref{algo:hhQR}.
The highest computational burden at each iteration falls on \cref{algo:hhQR} line 6, which requires Level-2 BLAS operations when computed efficiently. \par

In contrast, BQR replaces this step with Level-3 BLAS operations by partitioning $\bb{A}$ into blocks of columns.
Let $\bb{A} = [\bb{C}_1 \cdots  \bb{C}_N]$ where $\bb{C}_1,\cdots,\bb{C}_{N-1}$ are each $m$-by-$r$, and $\bb{C}_N$ holds the remaining columns.
The $k^{th}$ block, $\bb{C}_k$, is transformed with HQR (\cref{algo:hhQR}), and the WY representation of these $r$ successive HH transformations is constructed as in \cref{algo:buildWY}.
We write the WY update as
\begin{equation}
	\bb{X}_k = \bb{I}_m -\bb{W}_{k}\bb{Y}_{k}^{\top} = \bb{P}_k^{(1)}\cdots\bb{P}_{k}^{(r)}.
\end{equation}
Thus far, \cref{algo:hhQR,algo:buildWY} are rich in Level-2 BLAS operations.
Next, $\bb{I} -\bb{Y}_{k}\bb{W}_{k}^{\top}$ is applied to $[\bb{C}_2 \cdots  \bb{C}_N]$ with two Level-3 BLAS operations as shown in line 5 of \cref{algo:blockHQR}.
BQR performs approximately $1-\cO(1/N)$ fraction of its FLOPs in Level-3 BLAS operations (see section 5.2.3 of \cite{golub2013matrix}), and can reap the benefits from the accelerated bFMA feature of TensorCores. 
Note that BQR does require strictly more FLOPs when compared with HQR, but these additional FLOPs are negligible in standard precision and do not impact the numerical stability.
A pseudoalgorithm for BQR is shown in \cref{algo:blockHQR} where we assume that $n=Nr$ to make our error analysis in \cref{sec:BQRerr} simple.
In practice, an efficient implementation might require $r$ to be a power of two or a product of small prime factors and result in a thinner $N^{th}$ block compared with the rest. 
This discrepancy is easily fixed by padding the matrix with zeros, a standard procedure for standard algorithms like the Fast Fourier Transform (FFT).
For any variable $x$ in $\{\bb{X},\bb{W}, \bb{Y}, \bb{z}, \beta, \bb{v}, \bb{P}\}$,  $x_k^{(j)}$ corresponds to the $j^{th}$ update for the $k^{th}$ block.
\begin{algorithm2e}
	\DontPrintSemicolon 
	\KwIn{$\bb{A}\in\R^{m \times n}$, $r\in\R$ where $r < n$.}
	\KwOut{$\bb{Q},\bb{R}$}
	$N=\frac{n}{r}$\\
	\tcp{Let $\bb{A} = [\bb{C}_{1} \cdots  \bb{C}_{N}]$ where all blocks except $\bb{C}_{N}$ are $m$-by-$r$ sized.}
	\For{$i=1:N$}{
		$\bb{V}_{i},\bm{\beta}_i,\bb{C}_{i}\gets$ {\tt hhQR}($\bb{C}_{i}$)\tcc*{\Cref{algo:hhQR}}
		$\bb{W}_{i}\gets $ {\tt buildWY}$(\bb{V}_{i},\bm{\beta}_i)$ \tcc*{\Cref{algo:buildWY}}
		$[\bb{C}_{i+1}\cdots\bb{C}_{N}]$ -= $\bb{V}_{i} \left(\bb{W}_{i}^{\top}[\bb{C}_{i+1}\cdots\bb{C}_{N}]\right) $ \tcc*{update the rest: BLAS-3}
	}
	\tcp{$\bb{A}$ has been transformed into $\bb{R}=\bb{Q}^{\top}\bb{A}$.}
	\tcp{Now build $\bb{Q}$ using level-3 BLAS operations.} 
	$\bb{Q}\gets \bb{I}$\tcc*{$\bb{I}_m$ if full QR, and $\bb{I}_{m\times n}$ if thin QR.}
	\For{$i=N:-1:1$}{
		$\bb{Q}[(i-1)r+1:m,(i-1)r+1:n]$-= $\bb{W}_i \left(\bb{V}_i^{\top}\bb{Q}[(i-1)r+1:m,(i-1)r+1:n]\right)$
	}
	\Return $\bb{Q},\bb{A}$
	\caption{\label{algo:blockHQR} $\bb{Q},\bb{R}\gets {\tt blockHQR}(\bb{A}, r)$: Perform HH QR factorization of matrix $\bb{A}$ with column partitions of size $r$.}
\end{algorithm2e}
\subsubsection{BQR: Rounding Error Analysis}\label{sec:BQRerr}
We now present the basic structure for the rounding error analysis for \cref{algo:blockHQR}, which consists of: 1) HQR, 2) building the W factor, and 3) updating the remaining blocks with the WY representation.
We have adapted the analysis from \cite{Higham2002} to fit this particular variant, and let $\hat{\bb{Q}}_{BQR},\hat{\bb{R}}_{BQR}$ denote the outputs from \cref{algo:blockHQR}.
First, we analyze the error accumulated from updating $\bb{X}_k^{(j-1)}$ to $\bb{X}_k^{(j)}$, which applies a rank-1 update via the subtraction of the outer product $\hat{\bb{z}}_{k}^{(j)}\hat{\bb{v}}_{k}^{(j)\top}$.
Since $\bb{z}_{k}^{(j)} = \beta_k^{(j)}\bb{X}_{k}^{(j-1)}\bb{v}_{k}^{(j)}$, this update requires a single HH transformation on the right side in the same efficient implementation that is discussed in \cref{eqn:effH},
\begin{equation}
\hat{\bb{X}}_k^{(j)} = \hat{\bb{X}}_k^{(j-1)} - \fl(\hat{\beta}_k^{(j-1)}\hat{\bb{X}}_k^{(j-1)}\hat{\bb{v}}_k^{(j-1)})\hat{\bb{v}}_k^{(j)\top} = \hat{\bb{X}}_k^{(j-1)}(\bb{P}_k^{(j)}+\Delta \bb{P}_k^{(j)}), \label{eqn:Xupdate}
\end{equation}
where $\|\Delta \bb{P}_k^{(j)}\|_F \leq \tilde{\gamma}_{m-(k-1)r}$.
Since $\hat{\bb{X}}_k^{(1)} = \bb{I} - \hat{\beta}_k^{(1)}\hat{\bb{v}}_k^{(1)}\hat{\bb{v}}_k^{(1)\top} = \bb{P}_k^{(1)} + \Delta \bb{P}_k^{(1)}$, we can travel up the recursion relation in \cref{eqn:Xupdate} and use \cref{lem:3.7} to form
\begin{equation}
	\|\Delta \bb{X}_k^{(j)} \|_F := \|\hat{\bb{X}}_k^{(j)}-\bb{X}_k^{(j)}\|_F \leq j\tilde{\gamma}_{m-(k-1)r}. \label{eqn:deltX}
\end{equation}

\paragraph{HQR within each block: line 3 of \cref{algo:blockHQR}}
We apply \Cref{algo:hhQR} to the $k^{th}$ block, $\hat{\bb{X}}_{k-1}\cdots\hat{\bb{X}}_1\bb{C}_k$, which applies $r$ more HH transformations to columns that had been transformed by $(k-1)$ WY transformations in prior iterations.
The upper trapezoidal factor that results from applying HQR to $\bb{C}_{k}^{((k-1)r)}$ corresponds to columns $(k-1)r+1$ through $kr$ of $\hat{\bb{R}}_{BQR}$, and applying \cref{lem:3.7,lem:19.3} yields
\begin{equation*}
	\|\hat{\bb{R}}_{BQR}[:,j]-\bb{R}[:,j]\|_2 \leq r\tilde{\gamma}_{m}\|\hat{\bb{X}}_{k-1}\cdots\hat{\bb{X}}_1^{\top}\bb{C}_k[:,j]\|_2,\;\; j=(k-1)r+1:kr.
\end{equation*}
\paragraph{Build WY at each block: line 4 of \cref{algo:blockHQR}}
We now calculate the rounding errors incurred from building the WY representation when given a set of HH vectors and constants as shown in \cref{algo:buildWY}.
Since the columns of $\hat{\bb{Y}}_k$ are simply $\{\hat{\bb{v}}_k^{(j)}\}$ built in \cref{algo:hhQR} the errors for forming these are shown in \cref{eqn:vbound} where $m$ should be replaced by $m-(k-1)r$.
The HH constants, $\hat{\beta}_k^{(j)}$ are bounded by \cref{eqn:beta3} modified similarly. 
Thus, $\bb{z}_k^{(j)}$ is the only newly computed quantity. 
Using \cref{eqn:deltX,eqn:vbound,eqn:beta3}, we find
\begin{align*}
\|\Delta \bb{z}_k^{(j)}\|_2 &= \|\Delta\bb{X}_k^{(j-1)}\hat{\beta}_k^{(j)}\hat{\bb{v}}_k^{(j)} \|_2 \leq \|\Delta\bb{X}_k^{(j-1)}\|_2 \|\hat{\beta}_k^{(j)}\hat{\bb{v}}_k^{(j)}\|_2  \leq \|\Delta\bb{X}_k^{(j)-1}\|_F\|\hat{\beta}_k^{(j)}\hat{\bb{v}}_k^{(j)}\|_2 \\
& \leq \left((1+(j-1)\tilde{\gamma}_{m-(k-1)r})(1 + \tilde{\gamma}_{m-(k-1)r})-1\right) \| \beta_k^{(j)}\bb{v}_k^{(j)}\|_2 \leq j\tilde{\gamma}_{m-(k-1)r}\|\bb{z}_k^{(j)}\|_2.
\end{align*}
Componentwise bounds follow immediately, and are summarized in \cref{lem:BQR-build}.
\begin{lemma}\label{lem:BQR-build}
	Consider the construction of the WY representation for the $k^{th}$ partition of matrix $\bb{A}\in\R^{m\times n}$ given a set of HH constants and vectors, $\{\beta_k^{(j)}\}_{j=1}^r$ and $\{\bb{v}_{k}^{(j)}\}$ via \cref{algo:buildWY}.
	Then, 
	\begin{equation}
		\hat{\bb{z}}_{k}^{(j)} = \bb{z}_{k}^{(j)} + \Delta \bb{z}_{k}^{(j)},\;\; |\Delta \bb{z}_{k}^{(j)}| \leq j\tilde{\gamma}_{m-(k-1)r} |\bb{z}_{k}^{(j)}|,\;\; \|\Delta \bb{z}_k^{(j)}\|_2 \leq j\tilde{\gamma}_{m-(k-1)r}\|\bb{z}_k^{(j)}\|_2.\label{eqn:BQR-z}
	\end{equation}
\end{lemma}
Most importantly, this shows that constructing the WY update is just as numerically stable as applying successive HH transformations (see Section 19.5 of \cite{Higham2002}).

\paragraph{Update blocks to the right: line 5 of \cref{algo:blockHQR}}
We now consider applying $\bb{X}_{k}:=\bb{I}-\bb{W}_k\bb{Y}_k^{\top}$ to some matrix, $\bb{B}$.
In practice, $\bb{B}$ is the bottom right submatrix, $[\bb{C}_{k+1}\cdots \bb{C}_{N}][(k-1)r+1:m,:]$.
We can apply \cref{eqn:deltX} directly to the columns of $\bb{B}$, 
\begin{align}
	\|\fl(\hat{\bb{X}}_k \bb{B}[:,j])\|_2 = \|\fl(\hat{\bb{X}}_k^{(r)} \bb{B}[:,j])\|_2 \leq r\tilde{\gamma}_{m-(k-1)r} \|\bb{B}[:,j]\|_2.
\end{align}
A normwise bound for employing a general matrix-matrix multiplication operation is stated in section 19.5 of \cite{Higham2002}.
\paragraph{Multiple WY updates: line 8-9 of \cref{algo:blockHQR}}
All that remains is to consider the application of successive WY updates to form the QR factorization computed with BQR denoted as $\bb{Q}_{BQR}$ and $\bb{R}_{BQR}$. 
We can apply \cref{lem:3.7} directly by setting $\bb{X}_{k}:= \bb{I}-\bb{W}_{k}\bb{Y}_{k}^{\top}$ and consider the backward errors for applying the sequence to a vector, $\bb{x}\in\R^{m}$, as we did for \cref{lem:19.3}. 
Since $\bb{X}_{k}=\bb{P}_{(k-1)r+1}\cdots\bb{P}_{kr}$, is simply a sequence of HH transformations, it is orthogonal, i.e. $\|\bb{X}_{k}\|_2=1$.
We only need to replace with $\bb{x}$ with $\bb{A}[:,i]$'s to form the columnwise bounds for $\bb{R}_{BQR}$, and apply the transpose to $\hat{\bb{e}}_i$'s to form the bounds for $\bb{Q}_{BQR}$.
Then, 
\begin{align}
\left|\left|\prod_{k=1}^N (\bb{X}_{k} + \Delta \bb{X}_{k})- \prod_{k=1}^N\bb{X}_{k} \right|\right|_F &\leq\left(-1+\sum_{k=1}^N (1+r\tilde{\gamma}_{m-(k-1)r})\right) \leq rN\tilde{\gamma}_m \equiv n\tilde{\gamma}_m ,\label{eqn:BQR-mp}\\
\|\hat{\bb{Q}}_{BQR}-\bb{Q}\|_F&\leq n^{3/2}\tilde{\gamma}_m. \label{eqn:BQR}
\end{align}
We can also form the normwise bound for the $j'^{\ th}$ column of $\hat{\bb{Q}}_{BQR},\hat{\bb{R}}_{BQR}$. 
If we let $k' = \lceil j'/r\rceil^{th}$, then the $j'^{\ th}$ column is the result of applying $k'-1$ WY updates and an additional HQR. 
Applying \cref{lem:3.7} yields 
\begin{align}
\|\Delta \bb{R}_{BQR}[:,j']\|_2 \leq rk'\tilde{\gamma}_{m} \|\bb{A}[:,j']\|_2,&\;\; \|\Delta \bb{R}_{BQR}\|_F \leq n\tilde{\gamma}_{m} \|\bb{A}\|_F\\
\|\Delta \bb{Q}_{BQR}[:,j']\|_2 \leq rk'\tilde{\gamma}_{m},&\;\;\|\Delta \bb{Q}_{BQR}\|_F = r\tilde{\gamma}_{m}\sum_{j=1}^n \lceil j/r\rceil = n^{3/2}\tilde{\gamma}_{m}.\label{eqn:BQRmat}
\end{align}
and near orthogonality of the $\bb{Q}$ factor is still achieved.
\paragraph{BQR and HQR error bound comparison}
BQR under exact arithmetic is equivalent to HQR, and it is often referred to as the level-3 BLAS version of HQR. 
Furthermore, the error analysis of this section shows that BQR is as numerically stable as HQR despite requiring more FLOPs.
In fact, many linear algebra libraries such as LAPACK use a variant of BQR as the QR factorization algorithm (see {\tt dgeqrf} of \cite{LAPACK}).
The primary goal of the analysis presented in this section is to provide the basic skeleton for the standard BQR rounding error analysis to make the generalization to mixed precision settings in \cref{sec:mpanalysis} easier.
Readers should refer to \cite{golub2013matrix,Higham2002} for full details.
\subsection{Block HQR with partitioned rows : Tall-and-Skinny QR (TSQR)}\label{sec:TSQR}
Some important problems that require QR factorizations of overdetermined systems include least squares problems, eigenvalue problems, low rank approximations, as well as other matrix decompositions. 
Overdetermined systems with far more rows than columns are called tall-and-skinny.
Although Tall-and-Skinny QR (TSQR) broadly refers to block QR factorization methods with row partitions, we will discuss a specific variant of TSQR which is also known as the AllReduce algorithm \cite{Mori2012}.
In this paper, the TSQR/AllReduce algorithm refers to the most parallel variant of the block QR factorization algorithms discussed in \cite{Demmel2012}.
A detailed description and rounding error analysis of this algorithm can be found in \cite{Mori2012}, and we present a pseudocode for the algorithm in \cref{algo:par_tsqr}.
Our initial interest in this algorithm came from its parallelizable nature, which is particularly suitable to implementation on GPUs. 
Additionally, our numerical simulations (discussed in \cref{sec:NE}) show that TSQR can not only increase the speed but also achieve higher accuracy than the traditional HQR factorization in low precisions.
\subsubsection{TSQR/AllReduce Algorithm}
\Cref{algo:par_tsqr} partitions the rows of a tall-and-skinny matrix, $\bb{A}$. 
HQR is performed on each of those blocks and pairs of $\bb{R}$ factors are combined to form the next set of $\bb{A}$ matrices to be QR factorized. 
This process is repeated until only a single $\bb{R}$ factor remains, and the $\bb{Q}$ factor is built from all of the HH constants and vectors stored at each level.
The most gains from parallelization can be made in the initial level where the maximum number of independent HQR factorizations occur. 
Although more than one configuration of this algorithm may be available for a given tall-and-skinny matrix, the number of nodes available and the shape of the matrix eliminate some of those choices. 
For example, a 1600-by-100 matrix can be partitioned into 2, 4, 8, or 16 initial row-blocks but may be restricted by a machine with only 4 nodes, and a 1600-by-700 matrix can only be partitioned into 2 initial blocks.
Our numerical experiments show that the choice in the initial partition, which directly relates to the recursion depth of TSQR, has an impact in the accuracy of the QR factorization. \par

We refer to \emph{level} as the number of recursions in a particular TSQR implementation. 
An $L$-level TSQR algorithm partitions the original matrix into $2^{L}$ submatrices in the initial or $0^{th}$ level of the algorithm, and $2^{L-i}$ QR factorizations are performed in level $i$ for $i = 1 , \cdots, L$. 
The set of matrices that are QR factorized at each level $i$ are called $\bb{A}_j^{(i)}$ for $j = 1, \cdots, 2^{L-i}$, where superscript $(i)$ corresponds to the level and the subscript $j$ indexes the row-blocks within level $i$.
In the following sections, \cref{algo:par_tsqr} ({\tt tsqr}) will find a TSQR factorization of a matrix $\bb{A}\in\R^{m\times n}$ where $m \gg n$. 
The inline function {\tt qr} refers to \cref{algo:hhQR} and we use \cref{algo:hh_v2} as a subroutine of {\tt qr}.
\begin{algorithm2e}[H]
	\DontPrintSemicolon 
	\KwIn{$\bb{A}\in\R^{m \times n}$ where $m \gg n$, $L\leq\lfloor\log_2\left(\frac{m}{n}\right)\rfloor$, and $2^L$ is the initial number of blocks. }
	
	\KwOut{$\bb{Q}\in\R^{m \times n}$, $\bb{R} \in\R^{n\times n}$ such that 	$\bb{Q}\bb{R} = \bb{A}$.}
	$h \gets 2^{-L}m$ \tcp*{Number of rows.}
	\tcc{Split $\bb{A}$ into $2^L$ blocks. Note that level $(i)$ has $ 2^{L-i}$ blocks.}
	\For {$j = 1 : 2^L$}{
		$\bb{A}_j^{(0)} \gets \bb{A}[(j-1)h+1: jh, :]$ 
	}
	\tcc{Store HH vectors as columns of matrix $\bb{V}_j^{(i)}$, HH constants as components of vector $\bm{\beta}_j^{(i)}$, and set up the next level.}
	\For{$i = 0 : L-1$}{
		\tcc{The inner loop can be parallelized.}
		\For {$j = 1 : 2^{L-i}$ }{
			$\bb{V}_{2j-1}^{(i)}$, $\bm{\beta}_{2j-1}^{(i)}$, $\bb{R}_{2j-1}^{(i)} \gets{\tt qr}(\bb{A}_{2j-1}^{(i)})$ \;
			$\bb{V}_{2j}^{(i)}$, $\bm{\beta}_{2j}^{(i)}$, $\bb{R}_{2j}^{(i)} \gets{\tt qr}(\bb{A}_{2j}^{(i)})$\;
			\(\bb{A}_{j}^{(i+1)} \gets \begin{bmatrix}
			\bb{R}_{2j-1}^{(i)}\\
			\bb{R}_{2j}^{(i)}
			\end{bmatrix}\)
		}
	}
	$\bb{V}_{1}^{(L)}$, $\bm{\beta}_1^{(L)}$, $\bb{R}  \gets{\tt qr}(\bb{A}_{1}^{(L)})$ \tcp*{The final $\bb{R}$ factor is built.}
	$\bb{Q}_{1}^{(L)} \gets {\tt hh\_mult}(\bb{V}_{1}^{(L)}, I_{2n\times n})$\;
	\tcc{Compute $\bb{Q}^{(i)}$ factors by applying $\bb{V}^{(i)}$ to $\bb{Q}^{(i+1)}$ factors.}
	\For {$i = L-1 : -1 : 1$}{
		\For {$j = 1 : 2^{L-i}$}{
			\(\bb{Q}_{j}^{(i)} \gets {\tt hh\_mult}\left(\bb{V}_{j}^{(i)}, \begin{bmatrix}
			\tilde{\bb{Q}}_{\alpha(j), \phi(j)}^{(i+1)}\\
			\bb{0}
			\end{bmatrix}\right)\)
		}
	}
	$\bb{Q} \gets [];$\tcp*{Construct the final $\bb{Q}$ factor.}
	\For{$ j = 1 : 2^L $}{
		\(\bb{Q} \gets \begin{bmatrix}
		\bb{Q} \\
		{\tt hh\_mult}\left(\bb{V}_{j}^{(0)} , \begin{bmatrix}
		\tilde{\bb{Q}}_{\alpha(j), \phi(j)}^{(1)}\\
		\bb{0}
		\end{bmatrix} \right)
		
		\end{bmatrix}\)
		}
	\Return{$\bb{Q}$, $\bb{R}$}
	\caption{$\bb{Q},\bb{R}={\tt tsqr}(\bb{A}, L)$.  Finds a QR factorization of a tall, skinny matrix, $\bb{A}$. }
	\label{algo:par_tsqr}
\end{algorithm2e}
\paragraph{TSQR Notation}
We introduce new notation due to the multi-level nature of the TSQR algorithm.
In the final task of constructing $\bb{Q}$, $\bb{Q}_j^{(i)}$ factors are aggregated from each block at each level.
Each $\bb{Q}_j^{(i)}$ factor from level $i$ is partitioned such that two corresponding $\bb{Q}^{(i-1)}$ factors from level $i-1$ can be applied to them. 
The partition (approximately) splits $\bb{Q}_{j}^{(i)}$ into two halves, $[\tilde{\bb{Q}}_{j, 1}^{(i)\top} \tilde{\bb{Q}}_{j, 2}^{(i)\top}]^{\top}$.
The functions $\alpha(j)$ and $\phi(j)$ are defined such that $\bb{Q}_j^{(i)}$ is applied to the correct blocks from the level below: $\tilde{\bb{Q}}_{\alpha(j), \phi(j)}^{(i+1)}$.
For $j = 1 , \cdots, 2^{L-i}$ at level $i$, we need $j = 2(\alpha(j)-1) + \phi(j)$, where $\alpha(j) = \lceil \frac{j}{2}\rceil$ and $\phi(j) = 2 + j - 2\alpha(j) \in\{1,2\}$.
\Cref{Qdetails} shows full linear algebra details for a single-level ($L=1$, $2$ initial blocks) example.
The reconstruction of $\bb{Q}$ can be implemented more efficiently (see \cite{BDGJNS2014}), but the reconstruction method in \cref{algo:par_tsqr} is presented for a clear, straightforward explanation.
\subsubsection{Single-level Example}
\label{Qdetails}
In the single-level version of this algorithm, we first bisect $\bb{A}$  into $\bb{A}_1^{(0)}$ and $\bb{A}_2^{(0)}$ and compute the QR factorization of each of those submatrices.
We combine the resulting upper-triangular matrices (see below)  
which is QR factorized, and the process is repeated:
\[
\bb{A} = \begin{bmatrix}
\bb{A}_1^{(0)}\\
\bb{A}_2^{(0)}
\end{bmatrix} = \begin{bmatrix}
\bb{Q}_1^{(0)}\bb{R}_1^{(0)}\\
\bb{Q}_2^{(0)}\bb{R}_2^{(0)}
\end{bmatrix} = \begin{bmatrix}
\bb{Q}_1^{(0)} & \bb{0}\\
\bb{0} & \bb{Q}_2^{(0)}
\end{bmatrix} \begin{bmatrix}
\bb{R}_1^{(0)} \\
\bb{R}_2^{(0)}
\end{bmatrix} =\begin{bmatrix}
\bb{Q}_1^{(0)} & \bb{0}\\
\bb{0} & \bb{Q}_2^{(0)}
\end{bmatrix} \bb{A}_1^{(1)} =\begin{bmatrix}
\bb{Q}_1^{(0)} & \bb{0}\\
\bb{0} & \bb{Q}_2^{(0)}
\end{bmatrix} \bb{Q}_1^{(1)}\bb{R}.
\] 
The $\bb{R}$ factor of $\bb{A}_1^{(1)}$ is the final $\bb{R}$ factor of the QR factorization of the original matrix, $\bb{A}$. 
However, the final $\bb{Q}$ still needs to be constructed.
Bisecting  $\bb{Q}_1^{(1)}$ into two submatrices, i.e. $\tilde{\bb{Q}}_{1,1}^{(1)}$ and $\tilde{\bb{Q}}_{1,2}^{(1)}$, allows us to write and compute the product more compactly,  \[
\bb{Q}:=\begin{bmatrix}
\bb{Q}_1^{(0)} & \bb{0}\\
\bb{0} & \bb{Q}_2^{(0)}
\end{bmatrix} \bb{Q}_1^{(1)} =    \begin{bmatrix}
\bb{Q}_1^{(0)} & \bb{0}\\
\bb{0} & \bb{Q}_2^{(0)}
\end{bmatrix} \begin{bmatrix}
\tilde{\bb{Q}}_{1,1}^{(1)}\\
\tilde{\bb{Q}}_{1,2}^{(1)}
\end{bmatrix}= \begin{bmatrix}
\bb{Q}_1^{(0)}\tilde{\bb{Q}}_{1,1}^{(1)} \\ 
\bb{Q}_2^{(0)}\tilde{\bb{Q}}_{1,2}^{(1)}
\end{bmatrix}. \]
More generally, \cref{algo:par_tsqr} takes a tall-and-skinny matrix $\bb{A}$ and level $L$ and finds a QR factorization by initially partitioning $\bb{A}$ into $2^{L}$ row-blocks and includes the building of $\bb{Q}$.
For simplicity, we assume that $m$ is exactly $h2^{L}$ so that the initial partition yields $2^{L}$ blocks of equal sizes, $h$-by-$n$. 
Also, note that {\tt hh\_mult} refers to the action of applying multiple HH transformations given a set of HH vectors and constants, which can be performed by iterating line 6 of \cref{algo:hhQR}.
This step can be done in a level-3 BLAS operation via a WY update if \cref{algo:par_tsqr} was modified to store the WY representation at the QR factorization of each block of each level, $\bb{A}_j^{(i)}$. 

\subsubsection{TSQR: Rounding Error Analysis}
\label{sec:TSQRre}
The TSQR algorithm presented in \cref{algo:par_tsqr} is a divide-and-conquer strategy for the QR factorization that uses the HQR within the subproblems. 
Divide-and-conquer methods can naturally be implemented in parallel and accumulate less rounding errors.
For example, the single-level TSQR decomposition of a tall-and-skinny matrix $\bb{A}$ requires 3 total HQRs of matrices of sizes $\lfloor\log_{2}(\frac{m}{n})\rfloor$-by-$n$, $\lceil\log_{2}(\frac{m}{n})\rceil$-by-$n$, and $2n$-by-$n$.
The single-level TSQR strictly uses more FLOPs, but the dot product subroutines may accumulate smaller rounding errors (and certainly have smaller upper bounds) since they are performed on shorter vectors, and lead to a more accurate solution overall.
These concepts are elucidated in \cite{Mori2012} and we summarize the main results  in \cref{thm:moriTSQR}.

\begin{theorem}
	\label{thm:moriTSQR}
	Let $\bb{A}\in\R^{m\times n}$ with $m\geq n$ have full rank, $n$, and $\hat{\bb{Q}}_{TSQR}\in\R^{m\times n}$ and $\hat{\bb{R}}_{TSQR}\in\R^{n\times n}$ be the thin QR factors of $\bb{A}$ obtained via \cref{algo:par_tsqr} with $L$ levels. 
	Let us further assume that $m$ is divisible by $2^L$ and $n\tilde{\gamma}_{2^{-L}m}, n\tilde{\gamma}_{ 2n} \ll 1$.
	Then, 2-norm error bound for the $j^{th}$ column ($j=1:n$) of $\hat{\bb{R}}_{TSQR}$ and the Frobenius norm error bound for $\hat{\bb{Q}}_{TSQR}$ are
	\begin{align}
	\|\hat{\bb{R}}_{TSQR}[:,j]-\bb{R}[:,j]\|_2 &\leq n(\tilde{\gamma}_{2^{-L}m} + L\tilde{\gamma}_{ 2n})\|\bb{A}[:,j]\|_2,  \label{eqn:tsqrRA}\\
	\|\hat{\bb{Q}}_{TSQR}-\bb{Q}\|_F &\leq n^{3/2}(\tilde{\gamma}_{2^{-L}m} + L\tilde{\gamma}_{ 2n}).\label{eqn:tsqrQ}
	\end{align}
\end{theorem}

Note that the $n\tilde{\gamma}_{2^{-L}m}$ and $n\tilde{\gamma}_{ 2n}$ terms correspond to errors from applying HQR to the blocks in the initial partition and to the blocks in levels 1 through $L$ respectively.
We can easily replace these with analogous mixed precision terms and keep the analysis accurate.
Both level-2 and level-3 BLAS implementations will be considered in \cref{sec:mpanalysis}.
\paragraph{TSQR and HQR error bound comparison}
We compare the error bounds for HQR and TSQR. 
Consider the bounds for $\|\hat{\bb{Q}}-\bb{Q}\|_F$ in \cref{thm:feHQR,thm:moriTSQR}.
TSQR has a lower worst-case error bound than HQR when integers $m, n > 0$, and $L\geq0$ satisfy
\begin{equation*}
1\gg n^{3/2}\gamma_m \gg n^{3/2}(\gamma_{2^{-L}m}+L\gamma_{2n}).
\end{equation*}
Let us consider as an example the case when $\frac{m}{2^L}=2n$.
Then, the HQR bound is $2^L/(L+1)$ larger than the bound for TSQR with $L$ levels.
For example, in single precision, a HQR of a $2^{15}$-by-$2^6$ matrix results in an upper bound relative backward error ($\|\bb{A}-\hat{\bb{Q}}\hat{\bb{R}}\|_F/\|\bb{A}\|_F$) of $\approx${\tt1.002}, but a TSQR with $L=8$ is bounded by $\approx${\tt 3.516e-02}. 
This case exemplifies a situation in which accuracy is not guaranteed in HQR, but a relative error of $\approx 3.5\%$ is guaranteed when using TSQR. 
Note that these worst-case bounds are likely overestimates in practice
Now consider some $2^{20}$-by-$2^{12}$ matrix and QR factorizations performed with double precision.
The error bound for HQR is {\tt 1.686e-7}, whereas the error bound for TSQR with 12 levels is {\tt 5.351e-10}.
In general, we can conjecture that values of $L$ that can make $2^{-L}m$ and $2Ln$ much smaller than $m$, should produce a TSQR that outperforms HQR in worst-case scenarios, at least in uniform precision settings.
However, the range of matrix sizes that TSQR can accommodate decreases as $L$ grows larger.
Figure~\ref{fig:paramspace} shows the matrix sizes HQR, 2-level TSQR, and 4-level TSQR can accommodate as well as their respective error bounds on the $\bb{Q}$ factor.\par
\begin{wrapfigure}{l}{.45\textwidth}
	\centering
	\includegraphics[width=.45\textwidth]{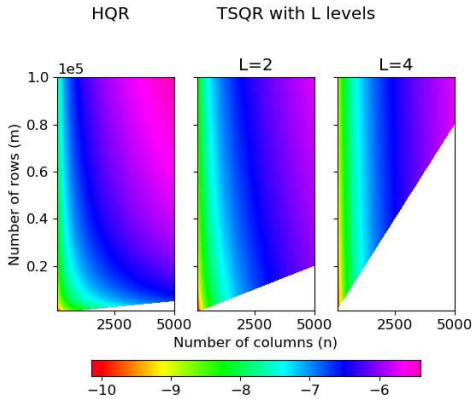}
	\caption{\label{fig:paramspace} Non-white space indicates allowable matrix sizes for each scheme, and color map represents error bounds for $\|\bb{\Delta Q}\|_F$ for uniform precision error analysis when using double precision arithmetic in $\log_{10}$ scale.}
	\vspace{-10pt}	
\end{wrapfigure}
\section{Mixed precision error analysis}\label{sec:mpanalysis}
In this section, we consider three different mixed precision settings for the QR factorization, all of which take in a matrix $\bb{A}$ stored in low precision and return $\bb{Q},\bb{R}$ both represented in low precision. 
First, we consider a trivial mixed precision setting where HQR, BQR, and TSQR are computed in high precision after casting up the input matrix at the beginning, and casting down the resulting high precision factors to low precision. 
Then in \cref{sec:mp-3}, we modify BQR and TSQR to utilize level-3 BLAS operations and TensorCore bFMAs for the matrix product subroutines. 
Finally, we impose \cref{assump:mp} in \cref{sec:mp-2} to see how a mixed precision inner product impacts HQR, BQR, and TSQR when applied in level-2 BLAS operations.

\paragraph{Backward error of casting down vectors} First, consider
casting down a vector  $\bb{x}\in\F_h^{(m)}$.
The componentwise forward error is, \[\text{\tt castdown}_{l}(\bb{x}) = \bb{x} + \Delta {\bb{x}},\;\; |\Delta\bb{x}| < u^{(l)}|\bb{x}|.\]
We use this to represent the backward error of a casting down a vector with a linear transformation, $\bb{I}^{(l)}:=\bb{I} +\bb{E}\in\R^{m\times m}$, a diagonal perturbation of the identity.
We write,
\begin{equation}
\bb{x}^{(l)} := \text{\tt castdown}(\bb{x}^{(h)}) = \bb{I}^{(l)}\bb{x}^{(h)} = (\bb{I}+\bb{E})\bb{x}^{(h)} = \bb{x}^{(h)}+\Delta \bb{x},
\end{equation}
where $|\Delta \bb{x}| \leq u^{(l)} |\bb{x}^{(h)}|$ and  $\|\Delta \bb{x}\|_2 \leq u^{(l)} \|\bb{x}^{(h)}\|_2$.
Thus, $\bb{E} = \Delta \bb{x x}^{\top}/\|\bb{x}\|_2^2$ and we can use the same argument as in \cref{eqn:outer} to form a backward matrix norm bound, 
\begin{equation}
\|\bb{E}\|_F\leq u^{(l)}. \label{eqn:castdown}
\end{equation}

\paragraph{Casting down after HQR in high precision} Let us consider the trivial case of carrying out HQR in high precision and casting down at the very end.
This is useful for the analysis of mixed precision
block algorithms as will be shown in \cref{sec:mp-3}.
If the two floating point types $\F_{l}$ and $\F_{h}$ satisfy $\F_{l}\subseteq \F_{h}$
and the matrix to be factorized is stored with low precision numbers, $\bb{A}\in\F_{l}^{m\times n}$, then casting up adds no rounding errors.
Therefore, we can directly apply the analysis that culminated in \cref{thm:feHQR}, and we only consider the columnwise forward error in the $\bb{Q}$ factor.
Then, the $j^{th}$ column of $\hat{\bb{Q}}_{HQR} = \bb{Q} + \Delta \bb{Q}_{HQR}$ is bounded normwise via $\|\Delta \bb{Q}_{HQR}[:,j]\|_2 \leq n\tilde{\gamma}_{m}^{h},$ and incurs an extra rounding error when $\hat{\bb{Q}}_{HQR}\in\F_{h}^{m\times n}$ is cast down to $\F_{l}^{m\times n}$.
Using this in \cref{lem:3.7} to analyze the forward norm error for the $j^{th}$ column of the $\bb{Q}$ factor computed with \cref{algo:hhQR} yields
\begin{equation}
	\|(\text{\tt castdown}(\hat{\bb{Q}}_{HQR})- \bb{Q})[:,j]\|_2 = \|(\bb{I}^{(l)}\hat{\bb{P}}_{1}\cdots\hat{\bb{P}}_{n}-\bb{P}_{1}\cdots\bb{P}_{n})\hat{\bb{e}}_j\|_2 \leq u^{(l)}+n\tilde{\gamma}_m^{(h)} + nu^{(l)}\tilde{\gamma}_m^{(h)}.\label{eqn:HQR-mp}
\end{equation}
The final castdown operation increases the upper bound by $u^{(l)}$ and the size of $\bb{A}$ has no impact on this extra rounding error.
Applying this trivial mixed precision setting to BQR and TSQR would simply increases the error bound by approximately $u^{(l)}$ all the while taking an even longer time than the high precision implementation due the extra cast down and cast up operations.
Therefore, we do not analyze the rounding error analysis of this mixed precision variant of BQR and TSQR.
However, we will use this mixed precision HQR as a subroutine of the mixed precision BQR and TSQR in the following section.
\subsection{Round down at inner product: level-2 BLAS mixed precision setting}\label{sec:mp-2}
Recall that HQR, BQR, and TSQR all rely on HH transformations in one way or another, and implementations of HH transformations are expressed by \cref{eqn:effH}.
This implementation capitalizes on the rank-1 update structure of HH transformations where the predominant share of FLOPs is spent on an inner product, and computing the HH vector and constant also rely heavily on inner products.
Therefore, nearly all of the computational tasks for \cref{algo:hhQR,algo:blockHQR,algo:par_tsqr} are attributed to the inner product, which is important in other linear algebra tools such as projections, matrix-vector, and matrix-matrix multiply.
Consequently, we return to \cref{assump:mp}, where every inner product is cast down to the lower precision as shown in \cref{eqn:aftercd}. 
We denote HQR, BQR, and TSQR computed with \cref{assump:mp} with {\tt mpHQR2}, {\tt mpBQR2}, and {\tt mpTSQR2}, where the {\tt 2} represents the mixed precision procedure computed at a level-2 BLAS operation.
\subsubsection{HQR round down at inner product: {\tt mpHQR2}}
Consider forming an HH transformation that zeroes out $\bb{x}\in\R^m$ below the the $i^{th}$ element. 
We need to compute $\sigma$, $\beta$, $\tilde{\bb{v}}_1$, and $\bb{v}$ as defined in \cref{sec:HQR},
\begin{align}
\fl(\sigma) &= \fl(-\rm{sign}(\bb{x}[1])\|\bb{x}\|_2) = \sigma + \Delta \sigma,\;\;|\Delta\sigma| \leq \left(\gamma_{2}^{(l)}+\gamma_{m}^{(h)}+\gamma_{2}^{(l)}\gamma_{m}^{(h)}\right)|\sigma|,\label{eqn:mpsigma}\\
\fl(\bb{v}'[1])& =\bb{v}'[1] + \Delta \bb{v}'[1] = (1+\dd^{(l)}) (\bb{x}[1]-\sigma-\Delta\sigma), \;\;|\Delta\bb{v}'[1]| \leq (\gamma_{3}^{(l)}+\tilde{\gamma}_{m}^{(h)})|\bb{v}'[1]| \label{eqn:mpv1}\\
\fl(\beta) &= \beta +\Delta \beta= (1+\dd^{(l)})\left(-\bb{v}'[1]/\hat{\sigma}\right), \;\; |\Delta\beta| \leq (\gamma_{5}^{(l)}+\tilde{\gamma}_{m}^{(h)})|\beta|, \label{eqn:mpbeta}\\
\fl(\bb{v}[j])	&= \bb{v}[j] + \Delta \bb{v}[j]\text{ where }|\Delta \bb{v}[j]|\leq 
(\gamma_{3}^{(l)} + \tilde{\gamma}_{m}^{(h)})|\bb{v}_j|,j=2:m-i+1 \label{eqn:mpv}. 
\end{align}
These bounds on $\Delta\sigma$, $\Delta \bb{v}'[1]$, $\Delta \beta$, and $\Delta \bb{v}[j]$ are computed by using the rules from \cref{lem:mp} on the analysis shown in \cref{sec:HQR}.
Using these, we can formulate the mixed precision version of \cref{eqn:applyP} where $\hat{\bb{y}}=\fl(\bb{P_vx})\in\R^m$ is implemented via \cref{eqn:effH}.
Note that the inner product $\hat{\bb{v}}^{\top}\bb{x}$ via \cref{assump:mp}, and all other operations are done in the lower precision.
Then, the transformed vector is bounded by
\begin{equation}
\hat{\bb{y}} = \bb{y}+\Delta \bb{y},\;\; \|\Delta \bb{y}\|_2 \leq (\gamma_{14}^{(l)} + \tilde{\gamma}_{m}^{(h)})\|\bb{y}\|_2.\label{eqn:mpdelty}
\end{equation}
Thus, a backward error can be formed using $\Delta \bb{P_v} = \Delta \bb{y}\bb{x}^{
	\top}/\|\bb{x}\|_2^2$ and by introducing the $\tilde{\gamma}$ notation for the low precision term,
\begin{equation}
\hat{\bb{y}} = (\bb{P_v} + \Delta \bb{P_v})\bb{x},\;\; \|\Delta \bb{P_v}\|_F\leq (\tilde{\gamma}_{10}^{(l)} + \tilde{\gamma}_{m}^{(h)}).\label{eqn:mpapplyP}
\end{equation}
Now, we form the error bounds for applying $n$ HH transformations to $\bb{x}$ using \cref{lem:3.7},
\begin{align}
\hat{\bb{z}} &= \fl(\bb{P}_1\cdots\bb{P}_n\bb{x})=\bb{Q} (\bb{x} +\Delta \bb{x}) = (\bb{Q} + \Delta \bb{Q})\bb{x},\\
\|\Delta \bb{y}\|_2 &\leq (\tilde{\gamma}_{10n}^{(l)}+n\tilde{\gamma}_m^{(h)})\|\bb{x}\|_2,\;\; \|\Delta \bb{Q}\|_F\leq (\tilde{\gamma}_{10n}^{(l)}+n\tilde{\gamma}_m^{(h)}).\label{eqn:mp19.3}
\end{align} 
The analogous mixed precision QR factorization error bounds are shown in \cref{thm:mpHQR}.
\begin{theorem}
	\label{thm:mpHQR}
	Let $\bb{A}\in\R^{m\times n}$ with $m\geq n$ have full rank, $n$. 
	Let $\hat{\bb{Q}}_{mpHQR2}\in\R^{m\times n}$ and $\hat{\bb{R}}\in\R^{n\times n}_{mpHQR2}$ be the thin QR factors of $\bb{A}$ obtained via \cref{algo:hhQR} with mixed precision FLOPs where inner products are computed in precision $h$ then cast down.
	All other operations are carried out in precision $l$.
	Then,
	\begin{align}
	\|\Delta \bb{R}_{mpHQR2}[:,j]\|_2&\leq (\tilde{\gamma}_{10n}^{(l)}+n\tilde{\gamma}_m^{(h)}) \|\bb{A}[:,j]\|_2,\;\; \|\Delta \bb{R}_{mpHQR2}\|_F\leq (\tilde{\gamma}_{10n}^{(l)}+n\tilde{\gamma}_m^{(h)}) \|\bb{A}\|_F \label{eqn:mpHQR2R}\\
	\|\Delta \bb{Q}[:,j]_{mpHQR2}\|_2&\leq (\tilde{\gamma}_{10n}^{(l)}+n\tilde{\gamma}_m^{(h)}),\;\; \|\Delta \bb{Q}_{mpHQR2}\|_F \leq n^{1/2} (\tilde{\gamma}_{10n}^{(l)}+n\tilde{\gamma}_m^{(h)})\label{eqn:mpHQR2Q}.
	\end{align}
\end{theorem}

Before commenting on the significance of \cref{thm:mpHQR}, we show that the same bounds hold for the BQR variant.
In the next sections we analyze using {\tt mpHQR2} instead of {\tt HQR} within \cref{algo:blockHQR,algo:par_tsqr}.

\subsubsection{BQR round down at inner product: {\tt mpBQR2}}
Now, we analyze \cref{algo:blockHQR} implemented with \cref{assump:mp}. 
At the $k^{th}$ block, we first apply the mixed precision HQR summarized in \cref{thm:mpHQR}.
Next, we construct the WY representation, where we can now use \cref{eqn:mpdelty,eqn:mpapplyP,lem:3.7} to form
\begin{equation}
\|\hat{\bb{X}}_{k}^{(l)}- \bb{X}_k\|_F = \|(\hat{\bb{P}}_k^{(1)}\cdots \hat{\bb{P}}_k^{(r)})-(\bb{P}_k^{(1)}\cdots \bb{P}_k^{(r)}))\|_F \leq \tilde{\gamma}_{10r}^{(l)} + r\tilde{\gamma}_{m}^{(h)}.
\end{equation}
Then, the 2-norm bound for the $j^{th}$ column of the $\bb{R}$ factor and the Frobenius norm bound for the orthogonal factor resulting from {\tt mpBQR2} are
\begin{align}
\|\hat{\bb{R}}_{mpBQR2}[:,j]\|_2 &= \|\hat{\bb{X}}_1\cdots\hat{\bb{X}}_N\bb{A}[:,j]\|_2\leq\left( N\tilde{\gamma}_{10r}^{(l)} + n\tilde{\gamma}_{m}^{(h)}\right)\|\bb{A}[:,j]\|_2,\\
\|\hat{\bb{Q}}_{mpBQR2}\|_F &\leq n^{1/2}\left(N\tilde{\gamma}_{10r}^{(l)} + n\tilde{\gamma}_{m}^{(h)}\right) \approx \left(1+\frac{10M_{l,h}}{m}\right)n^{3/2}\tilde{\gamma}_{m}^{(h)}. \label{eqn:mpBQR2}
\end{align}
Note that this error bound is of the same order as the error bound for {\tt mpHQR2}, shown in \cref{eqn:mpHQR2Q}.
\subsubsection{TSQR round down at inner product: {\tt mpTSQR2}}
Finally, we consider using \cref{assump:mp} in \cref{algo:par_tsqr}.
This corresponds to replacing every instance of $n\tilde{\gamma}_{m'}$ for $m'\in\{2n, 2^{-L}m\}$ in \cref{thm:moriTSQR} with $\tilde{\gamma}_{10n}^{(l)} + n\tilde{\gamma}_{m'}^{(h)}$.
We first consider the norm errors for the $j^{th}$ column of the $\bb{Q}$ factor computed by this mixed precision variant of \cref{algo:par_tsqr},
\begin{equation}
\|\hat{\bb{Q}}_{mpTSQR2}[:,j] -\bb{Q}[:,j]\|_2 \leq (L+1)\tilde{\gamma}_{10n}^{(l)} +n(\tilde{\gamma}_{2^{-L}m}^{(h)} + L\tilde{\gamma}_{ 2n}^{(h)}).\label{eqn:mptsqr2Qcol}
\end{equation} 
Then, the matrix norm error bound is 
\begin{align}
\|\hat{\bb{Q}}_{mpTSQR2}-\bb{Q}\|_F \leq n^{1/2}(L+1)\tilde{\gamma}_{10n}^{(l)} +n^{3/2}(\tilde{\gamma}_{2^{-L}m}^{(h)} + L\tilde{\gamma}_{ 2n}^{(h)})\\
\approx \left(1+ \frac{10M_{l,h}L}{2^{-L}m+ 2Ln}\right)n^{3/2}(\tilde{\gamma}_{2^{-L}m}^{(h)} + L\tilde{\gamma}_{ 2n}^{(h)}),\label{eqn:mptsqr2Q}
\end{align}
and contributes larger low precision rounding errors than in \cref{eqn:mpTSQR3}.
If the {\tt mpTSQR2} error bound were to outperform that of {\tt mpHQR2}, we now need integers $m, n > 0$, and $L\geq 0$ that satisfy
\begin{equation*}
1\gg n^{1/2}\left(\tilde{\gamma}_{10n}^{(l)} + n\tilde{\gamma}_{m}^{(h)}\right) \gg n^{1/2}\left((L+1)\tilde{\gamma}_{10n}^{(l)} +n(\tilde{\gamma}_{2^{-L}m}^{(h)} + L\tilde{\gamma}_{ 2n}^{(h)})\right).
\end{equation*}
In contrast to the analysis for uniform precision settings, large $L$ values do not necessarily reduce the error bounds of TSQR. 
While large $L$ can imply $m\gg 2^{-L}m+2Ln$, it is not always the case.
Although the theoretical error bounds do not give a clear indication of the worst-case performances of HQR and TSQR in mixed precision settings, TSQR outperformed HQR on ill-conditioned matrices within our numerical simulations.
These experiments are discussed in detail in \cref{sec:NE}.
\subsection{Round down at block-level: level-3 BLAS mixed precision setting}\label{sec:mp-3}
The mixed precision setting in this section is designed to meet the below requirements.
\begin{enumerate}
	\item Modify \Cref{algo:blockHQR,algo:par_tsqr} to maximize level-3 BLAS operations and use TensorCore bFMAs. 
	\item Apply \cref{eqn:HQR-mp} to all instances of HQR to the error analyses for BQR and TSQR in \cref{sec:algo}.
	\item Cast down quantities at every block/level and the insertion of low precision errors $u^{(l)}$ should be somewhat correlated to the number of blocks and levels. 
	\item Both input and output of the various QR factorization algorithms are given in the low precision. 
\end{enumerate}
TensorCore's bFMA can compute 
\begin{equation}
\hat{\bb{D}} =\fl_{TC}(\bb{C} + \bb{A}\bb{B}),\qquad \bb{C},\bb{D}\in\F_{\text{fp16}}^{4\times 4}\text{ or }\F_{\text{fp32}}^{4\times 4},\text{ and } \bb{A},\bb{B}\in\F_{\text{fp16}}^{4\times 4},\label{eqn:bFMA}
\end{equation}
and employ \emph{full} precision products and fp32 summation accumulate.
Note that the latest A100 GPUs are not restricted to fp16/fp32 in the TensorCore instructions \cite{nvdiaa100}.
Here, the \emph{full} precision multiplication is exact as explained in \cref{sec:background}.
In \cite{Blanchard2020}, the authors investigate all four possible matrix-matrix multiplication routines in TensorCores, which depend on whether $\bb{C}$ and $\bb{D}$ are computed in fp16 or fp32. 
They also note that matrices larger than $4$-by-$4$ can still be computed using this block FMA by accumulating matrix sums with $\bb{C}\in\F_{\text{fp32}}^{4\times 4}$.
Suppose that we aim to compute a fp16 matrix product of two fp16 matrices, $\bb{X}\in\F_{(fp16)}^{m\times p}$, $\bb{Y}\in\F_{(fp16)}^{p\times n}$, and $\bb{Z}=\bb{XY}\in\F_{\text{fp16}}^{m\times n}$.
We pad $\bb{X},\bb{Y}$ with zeros so that all matrix dimensions are multiples of $4$ and the matrix product can be computed with the TensorCore block FMA.
Let $\bb{Q}_{[i,j]}:= \bb{Q}[4(i-1)+1:4i,4(j-1)+1:4j]$ refer to the $(i,j)^{th}$ $4$-by-$4$ block for any $\bb{Q}\in\{\bb{X},\bb{Y},\bb{Z}\}$.
Then, we compute $\bb{Z}_{[i,j]}$ via \[
\bb{Z}_{[i,j]} = \sum_{k=1}^{\lceil p/4\rceil} \bb{X}_{[i,k]} \bb{Y}_{[k,j]},
\]
where we use \cref{eqn:bFMA} by initializing with $\bb{A}^{(1)}:= \bb{X}_{[i,1]}$, $\bb{B}^{(1)}:= \bb{Y}_{[1,j]}$, and $\bb{C}^{(1)}:= \bb{0}_{4\times 4}$ and setting $\bb{A}^{(k)}:= \bb{X}_{[i,k]}$, $\bb{B}^{(k)}:= \bb{Y}_{[k,j]}$, and $\bb{C}^{(k)}:= \bb{D}^{(k-1)}$ for $k=2:\lceil p/4\rceil$.
By setting $\bb{C}^{(k)}, \bb{D}^{(k)}\in\F_{\text{fp32}}^{4\times 4}$ for $k>1$ and only casting down at the end via $\bb{Z}_{[i,j]} =$ fp16$(\bb{D}^{(\lceil p/4\rceil)})$, we maximize our use of fp32 arithmetic.
This computes the most accurate mixed precision matrix product routine possible using TensorCore bFMAs whose inputs and output are required to be stored in fp16.
For example, take $p=8$.
Then the $[i,j]^{th}$ $4$-by-$4$ block of the product is computed via,
\begin{align*}
\bb{D}^{(1)} &= \fl_{TC}(\bb{X}_{[i,1]} \bb{Y}_{[1,j]}),\quad\bb{D}^{(2)} = \fl_{TC}(\bb{X}_{[i,2]} \bb{Y}_{[2,j]} + \bb{D}^{(1)})\in\F_{\text{fp32}}^{4\times 4}\\
\bb{Z}_{[i,j]} &= \text{\tt castdown}(\bb{D}^{(2)})\in\F_{\text{fp16}}^{4\times 4}.
\end{align*}
Adapting the rounding error analysis in \cite{Blanchard2020} into this specific mixed precision matrix product setting yields the componentwise forward bound 
\begin{equation}
|\bb{Z}-\fl(\bb{Z})| \leq \left(u^{(\text{fp16})}+ \gamma_{p}^{(\text{fp32})}+u^{(\text{fp16})} \gamma_{p}^{(\text{fp32})}\right)|\bb{X}||\bb{Y}|.\label{eqn:bFMAerr}
\end{equation}

We denote BQR and TSQR computed via TensorCore bFMA's with {\tt mpBQR3} and {\tt mpTSQR3}, where the {\tt 3} represents the BLAS level-3 nature of this mixed precision setting.
\subsubsection{BQR round down at block level: {\tt mpBQR3}}\label{sec:mp-3b}
Consider the input matrix, $\bb{A}\in\F_l^{m\times n}$, partitioned into $N$ blocks of $r$ columns, $\bb{A}=[\bb{C}_1 \cdots \bb{C}_N]$ as in \cref{sec:BQR}.
\Cref{algo:mpBQR} shows a mixed precision variant of BQR that maximizes the use of bFMAs but uses high precision arithmetic for level-1 and 2 BLAS operations which are only a $\cO(1/N)$ fraction of the total number of FLOPs. 
Each block is cast up to compute a high precision HQR and to form the WY representation. 
The WY representation is then cast down to low precision since the bFMAs require low precision inputs for matrix products, and the $\bb{R}$ factor from the high precision HQR can be cast down to return a low precision $\bb{R}$ factor at the very end. 
Since the cast down operations for the $\bb{R}$ factor and the WY representations occur at every block, we can expect columnwise error bound for \cref{algo:mpBQR} to increase by approximately $Nu^{(l)}$ from the error bound for \cref{algo:blockHQR}.
\begin{algorithm2e}
	\DontPrintSemicolon 
	\KwIn{$\bb{A}$, $r$. \hfill\textbf{Output: }$\hat{\bb{Q}}_{mpBQR3}$,$\hat{\bb{R}}_{mpBQR3}$}
	$N=\frac{n}{r}$\tcc*{Let $\bb{A} = [\bb{C}_{1} \cdots  \bb{C}_{N}]$.}
	\For{$k=1:N-1$}{
			$\bb{V}_{k},\bm{\beta}_k,\bb{C}_{k}\gets$ {\tt hhQR}({\tt castup}($\bb{C}_{k}$))\tcc*{\Cref{algo:hhQR} in high precision.}
		$\bb{C}_{k}\gets ${\tt castdown }($\bb{C}_{k}$)\tcc*{Builds $\bb{R}$ factor in low precision.}
		$\bb{W}_{k}\gets $ {\tt buildWY}$(\bb{V}_{k},\bm{\beta}_k)$ \tcc*{\Cref{algo:buildWY} in high precision}
		$[\bb{V}_{k},\bb{W}_{k}]\gets ${\tt castdown}($[\bb{V}_{k},\bb{W}_{k}]$)\;
		$[\bb{C}_{k+1}\cdots\bb{C}_{N}]$ -= $\bb{V}_{k} \left(\bb{W}_{k}^{\top}[\bb{C}_{k+1}\cdots\bb{C}_{N}]\right) $ \tcc*{returned in low precision}
	}
	$\bb{Q}\gets \bb{I}$\tcc*{Build $\bb{Q}$: $\bb{I}_m$ if full QR, and $\bb{I}_{m\times n}$ if thin QR.}
	\For{$k=N:-1:1$}{
		\tcp{All updates are returned in low precision.}
		$\bb{Q}[(k-1)r+1:m,(k-1)r+1:n]$-= $\bb{W}_k \left(\bb{V}_k^{\top}\bb{Q}[(k-1)r+1:m,(k-1)r+1:n]\right)$
	}
	\Return $\bb{Q},\bb{A}$
	\caption{\label{algo:mpBQR} $\hat{\bb{Q}}_{mpBQR3},\hat{\bb{R}}_{mpBQR3}\gets {\tt mpBQR3}(\bb{A}, r)$: Perform a mixed precision variant of BQR on low precision $\bb{A}$ with column partitions of size $r$. $\hat{\bb{Q}}_{mpBQR3}$,$\hat{\bb{R}}_{mpBQR3}$, are returned in low precision. Operations in lines 7 and 10 require low precision inputs.}
\end{algorithm2e}

Since $\hat{\bb{W}}_{k},\hat{\bb{Y}}_k$'s in \cref{algo:buildWY} are computed in high precision and then cast down, the new low precision WY update is $\hat{\bb{X}}_{k}^{(l)} = \bb{I}-\bb{I}^{(l)}\hat{\bb{W}}_k\bb{I}^{(l)}\hat{\bb{V}}_k^{(\top)}$.
Consider applying $\hat{\bb{X}}_k^{(l)}$ to some matrix $\bb{B}$ stored in low precision using the TensorCore bFMAs.
We analyze a single column $\bb{b}_j:=\bb{B}[:,j] \in \F_l^{m-(k-1)r}$ even though this operation is done on $\bb{B}$ as a whole.
Let $\bb{I}^{(l)}\hat{\bb{W}}_k = (\bb{I}+\bb{E}_W)\hat{\bb{W}}_k$ and $\bb{I}^{(l)}\hat{\bb{Y}}_k = (\bb{I}+\bb{E}_Y)\hat{\bb{Y}}_k$, where $\bb{E}_W,\bb{E}_Y$ are diagonal and bounded componentwise by $u^{(l)}$.
Then,the Frobenius norm error of forming $\hat{\bb{X}}_{k}^{(l)}$ is,
\begin{align*}
	\|\hat{\bb{X}}_{k}^{(l)}- \bb{X}_{k}\|_F  &= \|-\left(\bb{I}+\bb{E}_W+\bb{E}_Y+ \bb{E}_W\bb{E}_Y\right)\hat{\bb{W}}_k\hat{\bb{Y}}_k^{\top} + \bb{W}_k\bb{Y}_k^{\top}\|_F,\\
	&\leq \left((1+\gamma_2^{(l)}+(u^{(l)})^2)r\tilde{\gamma}_{m-(k-1)r}^{(h)}+\gamma_2^{(l)}+(u^{(l)})^2\right)\|\bb{X}_k\|_F\\
	&\leq \tilde{\gamma}_2^{(l)} +r\tilde{\gamma}_{m-(k-1)r}^{(h)} + r\tilde{\gamma}_2^{(l)}\tilde{\gamma}_{m-(k-1)r}^{(h)}.
\end{align*}
Now, we consider the backward error of applying $\hat{\bb{X}}_{k}^{(l)}$ to $\bb{b}_j$ with the bFMA matrix product error bound from \cref{eqn:bFMAerr}.
The multiplication by $(\bb{I}^{(l)}\hat{\bb{Y}}_k)^{\top}$ yields backward error bounded by
\begin{equation*}
	\fl_{TC}((\bb{I}^{(l)}\hat{\bb{Y}}_k)^{\top}\bb{b}_j) = (\hat{\bb{Y}}_k+\Delta_{TC}\hat{\bb{Y}}_k)\bb{b}_j,\;\;|\Delta_{TC}\hat{\bb{Y}}_k| \leq u^{(l)}+\gamma_{\frac{m-(k-1)}{4}}^{(h)}+u^{(l)}\gamma_{\frac{m-(k-1)}{4}}^{(h)}|\hat{\bb{Y}}_k||\bb{b}_j|,
\end{equation*}
and the subsequent multiplication by $(\bb{I}^{(l)}\hat{\bb{W}}_k)$ and subtraction from $\bb{b}_j$ result in,
\begin{align*}
	\fl_{TC}(\hat{\bb{X}}_{k}^{(l)}\bb{b}_j) &= (\hat{\bb{X}}_{k}^{(l)}+\Delta^{(l)}\bb{X}_k)\bb{b}_j,\\
	|\Delta^{(l)}\bb{X}_k| &\leq \left(\gamma_2^{(l)}+\gamma_{1+\frac{m-(k-2)r}{4}}^{(h)}+\gamma_2^{(l)}\gamma_{1+\frac{m-(k-2)r}{4}}^{(h)}\right)\left(|\bb{b}_j|+|\bb{I}^{(l)}\hat{\bb{W}}_k||\bb{I}^{(l)}\hat{\bb{Y}}_k|^{\top}|\bb{b}_j|\right).
\end{align*}
Converting to a normwise error bound using the same logic from \cref{eqn:applyP,eqn:19.2c} results in
\begin{equation}
\|\fl_{TC}(\hat{\bb{X}}_{k}^{(l)}\bb{b}_j)-\bb{X}_k\bb{b}_j\|_2 \leq (\tilde{\gamma}_2^{(l)} +r\tilde{\gamma}_{m-(k-1)r}^{(h)} + r\gamma_2^{(l)}\tilde{\gamma}_{m-(k-1)r}^{(h)}) \|\bb{b}_j\|_2, 
\end{equation}
since the rounding errors from the bFMAs are small in comparison to the errors from casting down the WY representation built in high precision.
The corresponding matrix error bound is
\begin{equation}
	\|\fl_{TC}(\hat{\bb{X}}_{k}^{(l)})-\bb{X}_k\|_F \leq \tilde{\gamma}_2^{(l)} +r\tilde{\gamma}_{m-(k-1)r}^{(h)} + r\tilde{\gamma}_2^{(l)}\tilde{\gamma}_{m-(k-1)r}^{(h)}.\label{eqn:mpdeltX}
\end{equation}
We can finally compute the forward errors from implementing \cref{algo:mpBQR}.
Consider the $j^{th}$ column of the $\bb{Q}$ factor, which we denote with $\bb{q}_j:=\hat{\bb{Q}}_{mpBQR3}[:,j]$, and let $k = \lfloor j/r\rfloor$.
Invoking \cref{lem:3.7} with error bounds for $\fl_{TC}(\hat{\bb{X}}_k^{(l)})$'s in \cref{eqn:mpdeltX} results in columnwise error,
\begin{align}
	\|\Delta \bb{q}_j \|_2 &\leq -1 + \prod_{k'=1}^k (1+\tilde{\gamma}_2^{(l)})(1+r\tilde{\gamma}_{m-(k'-1)r}^{(h)})\\ 
	&\leq k\tilde{\gamma}_{2}^{(l)} + kr\tilde{\gamma}_m^{(h)} + k^2r\tilde{\gamma}_{2}^{(l)}\tilde{\gamma}_m^{(h)}, \label{eqn:mpBQRcol}
\end{align} 
where $\Delta \bb{q}_j = (\fl_{TC}(\hat{\bb{X}}_1^{(l)})\cdots\fl_{TC}(\hat{\bb{X}}_k^{(l)}) - \bb{X}_1\cdots\bb{X}_k )\hat{\bb{e}}_j.$
Summing over the columns to find a matrix norm error bound yields
\begin{equation}
	\|\hat{\bb{Q}}_{mpBQR}-\bb{Q}\|_F \leq n^{1/2}\left(\tilde{\gamma}_{N}^{(l)} + n\tilde{\gamma}_m^{(h)}\right),\label{eqn:mpBQR3err}
\end{equation}
where the summation of the third term in \cref{eqn:mpBQRcol} is swept under the tilde notation in $n^{1/2} \tilde{\gamma}_{N}^{(l)}$.
This bound shows that \cref{algo:mpBQR} only adds $n^{1/2}\tilde{\gamma}_{N}^{(l)}$ order errors to the bounds in \cref{eqn:BQRmat}.
Using that $u^{(l)}=M_{l,h}u^{(h)}$, this increase corresponds to a multiplicative factor shown below,
\begin{equation}
	n^{1/2}\tilde{\gamma}_{N}^{(l)} + n^{(3/2)}\tilde{\gamma}_m^{(h)} \approx \left(1+\frac{M_{l,h}}{rm}\right)n^{(3/2)}\tilde{\gamma}_m^{(h)}. \label{eqn:mpBQR3}
\end{equation}
In practice, we expect hardware specifications to restrict $r$ between $\cO(10)$ and $\cO(100)$ and fix finite options for $u_l$, $u_h$, and therefore $M_{l,h}$.
This helps narrow our analysis. 
In general, the loss in accuracy due to mixed precision computing is relatively small when the disparity in precision ($M_{l,h}$) is small in comparison to the block size, $mr$.
As we can assume $r$ to be fixed at $\cO(100)$ at the largest, the loss of accuracy attributed to mixed precision computing grows small as $m$ grows large, indicating that {\tt mpBQR3} can yield accuracy comparable to its high, uniform precision variant while still benefiting from speed-ups. 
Next, note that $r=1$ and $r=n$ both revert to HQR with only level-2 BLAS operations, and other values of $r$ determine the proportion of level-2 and level-3 BLAS operations.
The optimal blocksize only needs to be searched within the permissible ranges for $r$ and $M_{l,h}$.
Overall, the actual trade-off between speed-ups and accuracy from using mixed precision hardware for the QR factorization is an open question that can be tackled more comprehensively in future research.
Our analysis shows that the worst-case bounds depend on block width $r$, the dimension of the input matrix $m,n$, as well as hardware specificities. 

\subsubsection{TSQR round down at block level: {\tt mpTSQR3}}\label{sec:mp-3t}
Unlike BQR which is rich in level-3 BLAS operations, the variant of TSQR in \cref{algo:par_tsqr} uses none.
Therefore, we modify \cref{algo:par_tsqr} by replacing all instances of {\tt hh\_mult} with level-3 BLAS operations.
We omit presenting the exact algorithm for mixed precision variant of TSQR in this paper, but consider computing the HQR of each block in high precision and build and store the WY representation of the HH transformations in low precision as we did in lines (3-6) of \cref{algo:mpBQR}.
The low precision WY representation is then applied with TensorCore bFMAs when building the $\bb{Q}$ factor (lines 11-16 of \cref{algo:par_tsqr}).
\paragraph{Rounding Error analysis} The analysis in \cite{Mori2012} shows that each column of $\bb{Q}$ is transformed by $n$ HH transformations of length $2n$ from levels $L:-1:1$, and another set of $n$ HH transformations of length $2^{-L}m$ at level $0$.
Let us represent the WY representation at the $j^{th}$ block of level $i$ and its bFMA counterpart as $\bb{X}_j^{(i)}$ and $\fl_{TC}(\hat{\bb{X}}_j^{(i)})$.
Then, we can use \cref{eqn:mpdeltX} to form backward error  
\begin{equation}
	\|\fl_{TC}(\hat{\bb{X}}_j^{(i)})-\bb{X}_j^{(i)})\|_F \leq \tilde{\gamma}_2^{(l)} +n\tilde{\gamma}_{m'}^{(h)} + n\tilde{\gamma}_2^{(l)}\tilde{\gamma}_{m'}^{(h)}, \;\; m' = \begin{cases}
	2^{-L}m, &i=0\\
	2n, & i = 1 : L
	\end{cases}.
\end{equation}
We can now modify the analysis in \cite{Mori2012} by replacing $n\tilde{\gamma}_{2^{-L}m}$ and $n\tilde{\gamma}_{2n}$ with \[(1+\tilde{\gamma}_2^{(l)})(1+n\tilde{\gamma}_{2^{-L}m}^{(h)})-1,\quad\text{and}\quad (1+\tilde{\gamma}_2^{(l)})(1+n\tilde{\gamma}_{2n}^{(h)})-1,\]
and apply \cref{lem:3.7}.
Then, the factors formed by {\tt mpTSQR3} are denoted by $\hat{\bb{R}}_{mpTSQR3},\hat{\bb{Q}}_{mpTSQR3}$ and the error bounds for the $j^{th}$ column of the triangular factor and the orthogonal factor are
\begin{align*}
\|(\hat{\bb{R}}_{mpTSQR3}- \bb{R})[:,j]\|_2 &\leq \tilde{\gamma}_{L+1}^{(l)}+n\left(L\tilde{\gamma}_{2n}^{(h)}+\tilde{\gamma}_{2^{-L}m}^{(h)}\right)\|\bb{A}[:,j]\|_2\label{eqn:mpTSQR1},\\
	\|\hat{\bb{Q}}_{mpTSQR3} - \bb{Q}\|_F &\leq n^{1/2}\tilde{\gamma}_{L+1}^{(l)}+n^{3/2}\left(L\tilde{\gamma}_{2n}^{(h)}+\tilde{\gamma}_{2^{-L}m}^{(h)}\right).
\end{align*}
Converting the low precision rounding errors as a fraction of the TSQR error bound in \cref{eqn:tsqrQ} to quantify the impact of modifying \cref{algo:par_tsqr} to utilize bFMAs yields
\begin{equation}
	n^{1/2}\tilde{\gamma}_{L+1}^{(l)}+n^{3/2}\left(L\tilde{\gamma}_{2n}^{(h)}+\tilde{\gamma}_{2^{-L}m}^{(h)}\right) = \left(1+ \frac{M_{l,h}(L+1)}{n(2^{-L}m+2nL)}\right)n^{3/2}\left(L\tilde{\gamma}_{2n}^{(h)}+\tilde{\gamma}_{2^{-L}m}^{(h)}\right).\label{eqn:mpTSQR3}
\end{equation}
Like in \cref{eqn:mpBQR3}, the disparity in the two precisions, $M_{l,h}$ is compared against the original matrix size $m,n$ and the block size specifications derived from $L$.
Let us consider the shallowest, middle, and the deepest levels of TSQR that are possible given some matrix in $\R^{m\times n}$.
All three cases in \cref{table:mpTSQR3} show that {\tt mpTSQR3} on sufficiently large matrices may yield errors closer to the high precision implementation, and the optimal choice for $L$ depends on $m,n$. 
\begin{table}[H]
	\center
	\begin{tabular}{||c|c|c|c||} 
		\hline
		 Number of levels, $L$& $1$ & $\frac{1}{2}\log_2(m/n)$ & $-1+\log_2(m/n)$ \\ \hline
		$\frac{(L+1)}{n(2^{-L}m+2nL)}$&  $1/(n^2+m/4)$ &$1/\left(2n^2+\frac{m^{1/2}n^{3/2}}{\log_2(m/n)}\right)$ & $1/(2n^2)$\\ \hline
	\end{tabular}
	\caption{Error bounds for $\|\Delta \bb{Q}_{mpTSQR3}\|_F$ for varying $L$'s.} 
	\label{table:mpTSQR3}
\end{table} 
\vspace{-1cm}
Finally, the error bounds for the $\bb{Q}$ matrix formed from implementing all of the algorithms discussed in \cref{sec:mpanalysis,sec:algo} are summarized below in \cref{table:summary}.
\begin{table}[H]
	\center
	\begin{tabular}{|p{0.75in}||p{1.3in}|p{1.5in}|p{1.5in}|}\hline
 Algorithm& Uniform & Mixed BLAS-2 (\S\ref{sec:mp-2}) & Mixed BLAS-3 (\S\ref{sec:mp-3})\\
 \hline\hline
 HQR (\S\ref{sec:HQR}) & $n^{3/2}\tilde{\gamma}_{m}$ & $n^{1/2}(\tilde{\gamma}_{10n}^{(l)}+n\tilde{\gamma}_{m}^{(h)})$ & n/a \\\hline
 BQR& $n^{3/2}\tilde{\gamma}_{m}$ & $n^{1/2}\left(\tilde{\gamma}_{10N}^{(l)}+n\tilde{\gamma}_{m}^{(h)}\right)$ & $n^{1/2}\left(\tilde{\gamma}_{N}^{(l)} + n\tilde{\gamma}_m^{(h)}\right)$ \\\hline
 TSQR & $n^{3/2}(\tilde{\gamma}_{m2^{-L}}+L\tilde{\gamma}_{2n})$ &  \begin{tabular}{@{}c@{}}$n^{1/2}\left[\right.\tilde{\gamma}_{10(L+1)}^{(l)}$\\ $+n\left(L\tilde{\gamma}_{2n}^{(h)}+\tilde{\gamma}_{m2^{-L}}^{(h)}\right)\left.\right]$\end{tabular} & \begin{tabular}{@{}c@{}}$n^{1/2}\left[\right.\tilde{\gamma}_{L+1}^{(l)}$ \\ $+n\left(L\tilde{\gamma}_{2n}^{(h)}+\tilde{\gamma}_{2^{-L}m}^{(h)}\right)\left.\right]$\end{tabular}\\
 \hline
\end{tabular}
	\caption{Error bounds for $\|\Delta \bb{Q}\|_F$ for all algorithms in \cref{sec:algo,sec:mpanalysis} for a matrix with $m$ rows and $n$ columns. The TSQR algorithms are performed with $L$ levels, and the BQR algorithms are performed with $N$ blocks. Low precision is denoted with $l$, and high precision, $h$.} 
	\label{table:summary},
\end{table} 
\vspace{-1cm}

Unsurprisingly, the inner product mixed precision setting from \cref{sec:mp-2} yields higher error bounds (see \cref{thm:feHQR}) as it uses more low precision arithmetic than the settings used in \cref{sec:mp-3}. 
For example, the error bound for {\tt mpBQR3} of \cref{sec:mp-3b} yielded low precision errors $r$ times smaller than that of {\tt mpBQR2} of \cref{sec:mp-2} , as intermediate results are cast down more often in {\tt mpBQR2}.
Therefore, guarantees of numerical stability of {\tt mpBQR2} are limited to smaller matrix sizes when compared with those of {\tt mpBQR3} and BQR in high precision.
While it is technically possible that the low precision errors introduced from utilizing \cref{assump:mp} do not dominate the errors incurred in {\tt mpBQR2} and {\tt mpHQR2} when $m\gg M_{l,h}$ and can result in accuracy comparable to that of {\tt mpBQR3} and high precision BQR, our numerical results in \cref{sec:NE} show that {\tt mpHQR2} is already unstable at $m\approx M_{l,h}$.

\section{Numerical Experiments}\label{sec:NE}
We conducted several numerical experiments to confirm the validity of the error bounds formed in \cref{sec:mpanalysis} by varying matrix size for all algorithms, block sizes in {\tt mpBQR3}, and comparing {\tt mpHQR2} against {\tt mpTSQR2} with varying condition numbers.
We used Julia, a programming language which allows fp16 storage and {\tt castup} and {\tt castdown} operations between types in {fp16, fp32, fp64}, but no built-in fp16 arithmetic.
Therefore, we relied on using \cref{algo:simulate} for $f\in \text{OP} \cup\{{\tt dot\_product}\}$ to simulate \cref{assump:mp} and TensorCore bFMAs.\par

In \cref{sec:algo,sec:mpanalysis}, we gave the forward error bounds for $\bb{R}$ and $\bb{Q}$ separately. 
Since our numerical experiments instead measure a backward error, $\|\hat{\bb{Q}}\hat{\bb{R}}-\bb{A}\|_F$, and an orthogonal error, $\|\hat{\bb{Q}}^{\top}\hat{\bb{Q}}-\bb{I}\|_2$, we show how to convert general forward errors into those computed quantities.
Given $\|(\hat{\bb{R}}-\bb{R})[:,j]\|_2\leq \epsilon_R \|\bb{A}[:,j]\|_2$ and $\|\hat{\bb{Q}}-\bb{Q}\|_F\leq \epsilon_Q$,
\begin{align}
	\|(\hat{\bb{Q}}\hat{\bb{R}}-\bb{A})[:,j]\|_2 &\leq (\epsilon_R+\epsilon_Q + \epsilon_R\epsilon_Q)\|\bb{A}[:,j]\|_2,\;\; j=1:n,\quad \text{see \cite{Higham2002}},\\
	\|\hat{\bb{Q}}\hat{\bb{R}}-\bb{A}\|_F &\leq n^{1/2}(\epsilon_R+\epsilon_Q + \epsilon_R\epsilon_Q)\|\bb{A}\|_F, \label{eqn:QRA} \\
	\|\hat{\bb{Q}}^{\top}\hat{\bb{Q}}-\bb{I}\|_2 &\leq \|\hat{\bb{Q}}^{\top}\hat{\bb{Q}}-\bb{I}\|_F \simeq 2\epsilon_Q,\quad\text{see \cite{Mori2012}} \label{eqn:QQI}.
\end{align}
First, we tested \cref{algo:hhQR,algo:blockHQR,algo:par_tsqr,algo:mpBQR}, {\tt mpHQR2}, {\tt mpBQR2}, and {\tt mpTSQR2} for varying matrix sizes.
We increased the number of rows $m$ from $1000$ to $13949$, while keeping $n=250$, $r=63$, and $L=2$ and the test matrices were sampled from the standard normal distribution.
On the left plot of \cref{fig:sizemp3}, we see three clusters which each correspond to: top, \cref{assump:mp}; middle, TensorCore bFMAs; and bottom, uniform precision implementations in fp32.
The high precision and bFMA implementations scale similarly to each other when increasing the matrix size, whereas the \cref{assump:mp} variants grow unstable more quickly.
In addition, while HQR, BQR, and TSQR perform similarly in high precision and when using bFMAs, {\tt mpTSQR2} is less accurate by a quarter to a half order of magnitude in comparison to {\tt mpBQR2} and {\tt mpHQR2}.
The specifications for $m,n,L,M_{l,h}$ for this experiment derive the upper bound for $\|\Delta \bb{Q}_{mpTSQR2}\|_F$, \cref{eqn:mptsqr2Q}, to be larger than that of $\|\Delta \bb{Q}_{mpHQR2}\|_F$, \cref{eqn:mpHQR2Q}.
However, a more careful comparison of {\tt mpHQR2} and {\tt mpTSQR2} show that there exists a regime where {\tt mpTSQR2} can outperform {\tt mpHQR2}.
\begin{figure}[h!]
	\centering
	\vspace{-10pt}
	\includegraphics[width=\textwidth]{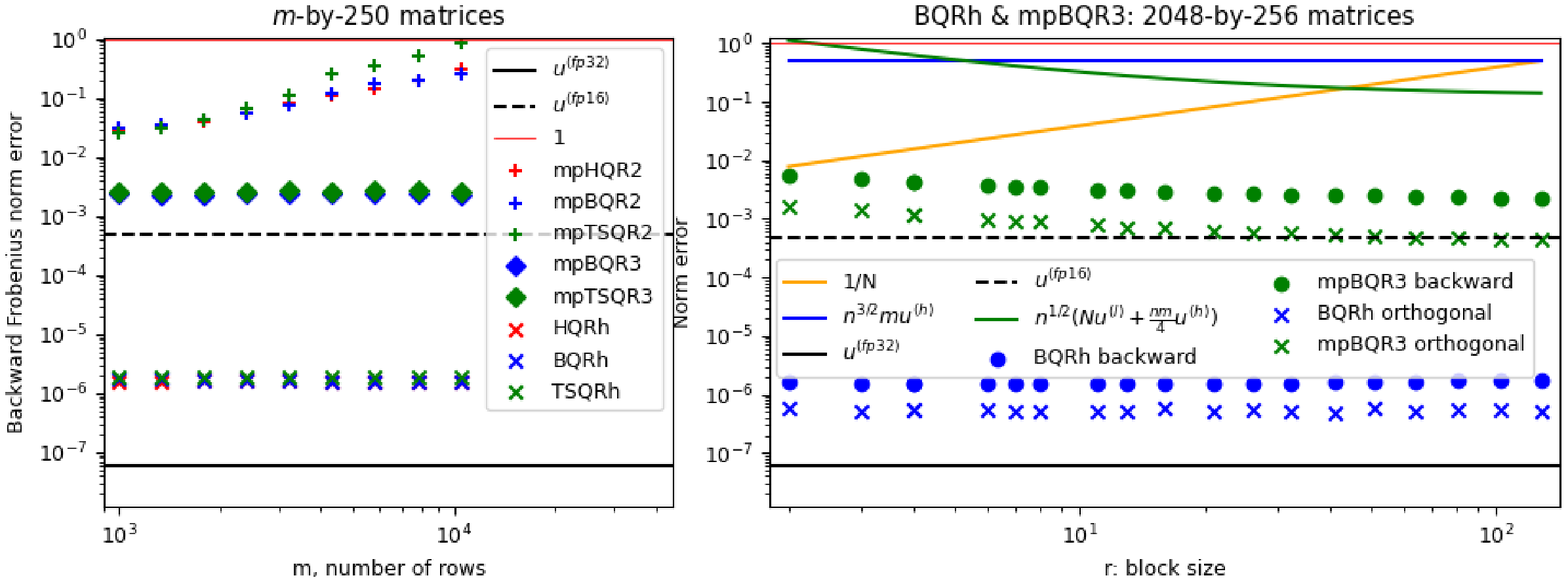}
	\vspace{-22pt}
	\caption{\label{fig:sizemp3}Left plot: Backward errors of HH QR factorization algorithms in \cref{sec:algo,sec:mpanalysis} with varying matrix sizes.
		Right plot: Norm errors of fp32 BQR and {\tt mpBQR3} for $2048$-by-$256$ matrices for varying block sizes.}
	\vspace{-10pt}
\end{figure} 

Next, we varied the block sizes for performing fp32 BQR and {\tt mpBQR3} on $2048$-by-$256$ sized matrices, which were chosen to yield error bounds below 1 for both algorithms.
The right plot of \cref{fig:sizemp3} shows the error bounds and the computed value for the backward error for the two algorithms where the block size $r$ varies from $2$ to $256$. 
The test matrices were generated following example from \cite{Blanchard2020} by setting $\bb{A}={\tt castdown}(\bb{Q}_1\bb{D}\bb{Q}_2)$ where $\bb{Q}_1\in\F_h^{m\times n}$, $\bb{Q}_2\in\F_h^{n\times n}$ are orthogonal and $\bb{D}=\mathrm{Diagonal}(\{\log_{10}(0),\cdots, \log_{10}(-3)\})\in\F_h^{n\times n}$.  
The high precision implementation yields backward error close to $u^{(fp32)}$ and {\tt mpBQR3} yields errors near $u^{(fp16)}$ that follows the downward trend suggested by \cref{eqn:mpBQR3}.
As block sizes increase, {\tt mpBQR3} grows more accurate. 
This trend correlates to $1/N$, the approximate fraction of FLOPs in {\tt mpBQR3} performed in high precision, marked in orange.
However, the rightmost data for {\tt mpBQR3} (corresponds to $r=n$), is still between 3 and 4 orders of magnitude less accurate than its high precision variant. 
Further studies that directly test speed-ups from bFMAs against the accuracy of {\tt mpBQR3} are needed to fully understand the potential uses for mixed precision QR algorithms.

Lastly, we compared a mixed precision variant of a communication-avoiding QR algorithm ({\tt mpTSQR2}) against a mixed precision variant of the HH QR algorithm ({\tt mpHQR2}) on a set of fixed-size matrices with varying condition numbers.
Note that an empirical comparison of the two algorithms implemented in fp64 arithmetic were reported in \cite{Mori2012}, and we omit the comparison against {\tt mpBQR2} since it performs very similarly to {\tt mpHQR2}.
Following example from \cite{Mori2012}, we used $m$-by-$n$ random matrices, $\bb{A}_{\alpha} = \bb{Q'}(\alpha \bb{E} + \bb{I})/\|\bb{Q'}(\alpha \bb{E} + \bb{I})\|_F$, where $\bb{Q'}\in\mathbb{R}^{m\times n}$ is orthogonal and $\bb{E}\in\R^{n\times n}$ is the matrix of $1$'s. 
We constructed $\bb{Q'}$ by computing the default QR factorization of matrix $\bb{\Omega}\in\F_{fp64}^{4000\times100}$ in Julia, which performs BQR with $r=36$ entirely in fp64 arithmetic, and elements of the random matrix $\bb{\Omega}$ were sampled from the uniform distribution over $[0,1]$.
By construction, $\bb{A}_{\alpha}$ has 2-norm condition number $n\alpha+1$. 
By varying $\alpha$ from {\tt 1e-4} to {\tt 1}, we varied the condition number from $1.1$ to $101$, and we generated $10$ samples for each value of $\alpha$.
Even though the condition number is not a part of our rounding error analysis, we use it in this experiment in an attempt to reach the ``worst-case'' scenario described by our deterministic error bounds in \cref{sec:mpanalysis}.
The relative backward error, $\|\hat{\bb{Q}}\hat{\bb{R}}-\bb{A}\|_F/\|\bb{A}\|_F$, was computed by casting up $\hat{\bb{Q}}$, $\hat{\bb{R}}$, and $\bb{A}$ to fp64 to compute the Frobenius norms.
Plugging in $m=4000$, $n=100$, $u^{(l)}=u^{(fp16)}$, $u^{(h)}=u^{(fp32)}$, and $c=1$ (for $\tilde{\gamma}$) into the error bounds for {\tt mpHQR2} combined with \cref{eqn:QRA,eqn:QQI} are approximately {\tt 1.179} and {\tt 1.146}.
These error bounds are \emph{relative} and these worst-case bounds do not guarantee errors below 100\%.
The TSQR bounds for the same parameters for $L=1:5$ are even larger, which indicates that stability is not guaranteed. 
The leftmost plot of \cref{fig:allTSQR} shows the backward errors of {\tt mpHQR2} increasing as the theoretical condition numbers of the generated random matrices increase, and these errors correspond to the error data on the vertical axis, $L=0$, of the middle plot.
In addition to the errors from {\tt mpHQR2}, Figure~\ref{fig:allTSQR} shows the errors from {\tt mpTSQR2s} of levels varying from $L=1$ to $L=5$, where each line represents the errors of HQR and variants of TSQR calculated from the same random test matrix.
Figure~\ref{fig:allTSQR} reveals two different trends for the errors as we deepen the complexity of the QR algorithm from {\tt mpHQR2} to {\tt mpTSQR2} with $L=5$, and these two trends are separated into the center and right plots. 
One trend occurs for matrices with smaller condition numbers, where {\tt mpHQR2} is stable, but {\tt mpTSQR2} with higher levels yield larger errors. 
Another trend occurs for matrices with higher condition numbers, where single-level and 2-level {\tt mpTSQR2} yield smaller errors than {\tt mpHQR2}. 
To make these two trends clear, we have used green segments to show when a TSQR with one more level yields smaller errors, and blue segments to show when a TSQR with another level yields larger errors.
In these cases, errors from {\tt mpTSQR2} with 3 or more levels are similar to or worse than their 2-level variants, but generally do not exceed those of {\tt mpHQR2} most of the times.
These results suggests that TSQR can outperform HQR in mixed precision settings, and particularly when HQR suffers from accumulated rounding errors.
\begin{figure}[h!]
	\centering
	\includegraphics[width=\textwidth]{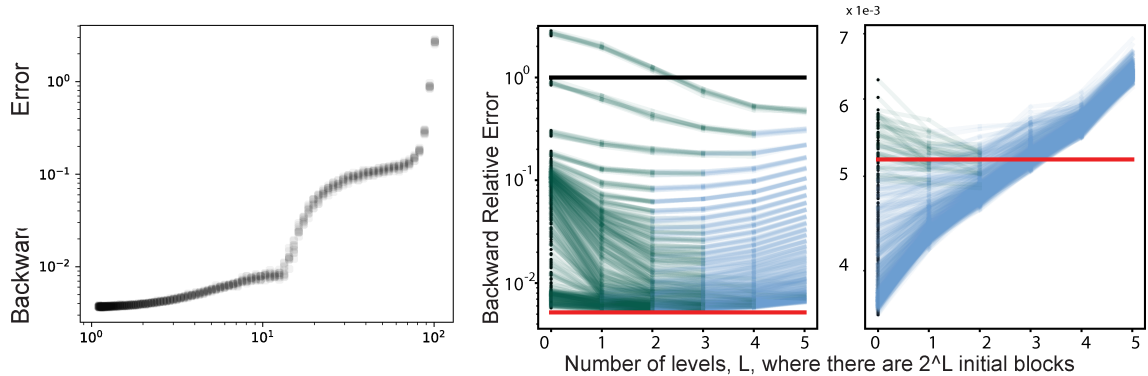}
	\vspace{-15pt}
	\caption{\label{fig:allTSQR} All plots show the backward relative error for 4000-by-100 sized test matrices. Left: {\tt mpHQR2} on condition numbers ranging from 1.1 to 101;  Middle: {\tt mpTSQR2} on condition numbers ranging from 5.3 to 101; Right:  {\tt mpTSQR2} on condition numbers ranging from 1.1 to 5.3.  The red line on both plots corresponds to value {\tt 5.2e-3},and the green/blue segments show a decrease/increase in error in comparison to using TSQR with one fewer level. }
	\vspace{-5pt}
\end{figure}

In conclusion, most of the experiments display the trends that error bounds in \cref{sec:algo,sec:mpanalysis} suggest, and bFMA variants perform in between the high precision and \cref{assump:mp} variants as expected.
Also, a special case is shown that demonstrate {\tt mpTSQR2} can outperform {\tt mpHQR2} in accuracy despite having higher error bounds.
All of the experiments showed that the actual errors were many orders of magnitude lower than the error bounds even when ill-conditioned, but this discrepancy varied for different mixed precision settings.
For example, backward and forward errors of {\tt mpBQR3} were \emph{only} 2-3 orders of magnitude below the error bounds, whereas the fp32 implementation of BQR yielded errors up to 6 orders of magnitude below the error bounds.
Although further studies with larger problem sizes and timings would be beneficial in developing an {\tt mpBQR3} with the optimal block size, $r$, our experiments confirm the intuition built from the error analysis in \cref{sec:mpanalysis}.
\section{Conclusion}
The development of GPUs that optimize low precision floating point arithmetic have accelerated the interest in half and mixed precision algorithms that naturally reduces the bandwidth and storage needs. 
Loss in precision, stability, and representable range offset for those advantages, but these shortcomings may have little to no impact in some applications.
It may even be possible to navigate around those drawbacks with algorithmic design. \par
We present the algorithm and standard error analysis of HQR and its blocked variants (BQR and TSQR), modify the algorithms to support two mixed precision settings, and performed error analyses that bound the mixed precision versions.
One mixed precision setting is that of NVIDIA's TensorCore bFMAs, and the other is an ad hoc setting that mimics the bFMAs at the level of inner products.
These two are presented to offer mixed precision arithmetic at both level-2 and 3 BLAS operations and can be applied to other linear algebra tools as well.
The new error bounds more accurately describe how rounding errors are accumulated in mixed precision settings.
For a given problem, available hardware, and some error tolerance, these bounds can be used to first narrow down which QR factorization algorithms are feasible. 
Then, the speed-ups from the hardware specifications can be considered next to choose the most appropriate settings within the algorithms (i.e. block size $r$ in BQR or number of levels, $L$, in TSQR).
We found that TSQR can outperform HQR under \cref{assump:mp} for ill-conditioned, extremely overdetermined cases even when the error bounds imply the opposite.
While an optimistic interpretation of this result would be that algorithms like TSQR are more robust against lower precision arithmetic, further research is needed to explore other divide-and-conquer methods that can harness parallel capabilities.
Meanwhile, we should rely on the error bounds formed in \cref{sec:mpanalysis}.
\section*{Acknowledgements}
The authors are grateful to the anonymous reviewers for their constructive criticisms, which improved the clarity of the presentation.
\bibliographystyle{siamplain}
\bibliography{library.bib,sans_library.bib,./report.bib}

\end{document}